\documentclass[a4paper,reqno,10pt]{amsart}

\usepackage{hyperref,amsthm}
\usepackage[noabbrev]{cleveref}

\usepackage{fouche/fouche}
\def\cLin{\cate{Lin}}
\usepackage{stmaryrd}

\title{Automata and coalgebras in categories of species}

\author{Fosco Loregian}
\email{fosco.loregian@gmail.com}
\address{Tallinn University of Technology}

\thanks{%
  F. Loregian was supported by the Estonian Research Council grant PRG1210. The paper in its present form is an extended version of the note published in the proceedings of \textsc{cmcs}\oldstylenums{2024} as \cite{loregian2024automata}. Some of these computations were suggested by Todd Trimble, who pointed out the existence of the `Euler' derivation and proposed \autoref{ex_alg_2}, which in turn suggested a simple description of $\Spc^\Lin$ and more examples, by analogy. The author is extremely grateful to Todd for his invaluable contribution and his mathematical generosity.
 }
\usepackage[greek,english]{babel}
\usepackage[cal=boondoxo]{mathalfa}

\NewDocumentCommand{\mlyob}{ m m m }{\mathop{\dsdtstile{#2,#3}{#1}}}
\NewDocumentCommand{\mreob}{ m m m }{\dsststile{#2,#3}{#1}}
\renewcommand{\de}{\hyperref[diff_of_spc]{\partial}}
\RenewDocumentCommand{\id}{ O{\clK} }{\text{id}_{#1}}
\RenewDocumentCommand{\Mly}{O{\clK}}{\hyperref[mly_n_mre]{\mathsf{Mly}}_{#1}}
\RenewDocumentCommand{\Mre}{O{\clK}}{\hyperref[mly_n_mre]{\mathsf{Mre}}_{#1}}
\NewDocumentCommand{\clMly}{O{\clK}}{\mathcal{M}\!\mathcal{l}\!\mathcal{y}_{#1}}
\NewDocumentCommand{\clMre}{O{\clK}}{\mathcal{M}\!\mathcal{r}\!\mathcal{e}_{#1}}
\NewDocumentCommand{\clMlyTens}{O{\clK}}{\clMly[#1]^\otimes}
\NewDocumentCommand{\clMreTens}{O{\clK}}{\clMre[#1]^\otimes}

\newcommand{\Spc}{{\hyperref[spc_n_Vspc]{\mathsf{Spc}}}}
\def\slice{/\!\!\!/}

\def\Day{\mathbin{\hyperref[ms_2]{\otimes_\mathrm{Day}}}}

\def\Top{\mathsf{Top}}
\def\textEM{Eilenberg--Moore\@\xspace}
\NewDocumentCommand{\homDay}{O{\firstblank} O{\firstblank}}{\{#1,#2\}_\mathrm{Day}}
\def\Lin{{\hyperref[es_3]{\fkL}}}
\setlist{itemsep=0pt}

\NewDocumentCommand{\pair}{ O{.5em} m m O{r} }{
\ar@<#1>[#4]^-{#2}\ar@<-#1>[#4]_-{#3}
}

\NewDocumentCommand{\adpair}{ O{.5em} m m O{r} }{
\ar@<#1>[#4]^-{#2}\ar@<-#1>@{<-}[#4]_-{#3} \ar@{}[r]|-\perp
}

\def\Due{[\oldstylenums{2}]}
\def\Uno{[\oldstylenums{1}]}
\def\Zero{[\oldstylenums{0}]}

\renewcommand{\sfP}{\hyperref[spc_n_Vspc]{\mathsf{P}}}
\def\Cyc{\hyperref[es_4]{\fkC\fky\fkc}}
\def\Exp{\hyperref[es_0]{\fkE}}

\usepackage{xspace}
\def\ie{i.e.\@\xspace}
\def\bang{\boldsymbol{!}}

\def\bfCat{\mathbf{Cat}}

\def\Arb{\bsA\bsr\bsb}
\def\Jet{\bsJ\bse\bst}
\usepackage[left=3cm,right=3cm]{geometry}
\usepackage{mlmodern}
\begin{document}
\begin{abstract}
	We study generalized automata (in the sense of Ad\'amek-Trnkov\'a) in Joyal's category of (set-valued) combinatorial species, and as an important preliminary step, we study coalgebras for its derivative endofunctor $\de$ and for the `Euler homogeneity operator' $L\circ\de$ arising from the adjunction $L\dashv\de\dashv R$. The theory is connected with, and in fact provides relatively nontrivial examples of, \emph{differential 2-rigs}, a notion recently introduced by the author putting combinatorial species on the same relation a generic (differential) semiring $(R,d)$ has with the (differential) semiring $\bbN\llbracket X\rrbracket$ of power series with natural coefficients. The desire to study categories of `state machines' valued in an ambient monoidal category $(\clK,\otimes)$ gives a pretext to further develop the abstract theory of differential 2-rigs, proving lifting theorems of a differential 2-rig structure from $(\clR,\de)$ to the category of $\de$-algebras on objects of $\clR$, and to categories of Mealy automata valued in $(\clR,\otimes)$, as well as various constructions inspired by differential algebra such as jet spaces and modules of differential operators. These theorems adapt to various `species-like' categories such as coloured species, $k$-vector species (both used in operad theory), linear species (introduced by Leroux to study combinatorial differential equations), M\"obius species, and others.
\end{abstract}
\maketitle

\setcounter{tocdepth}{1}
\tableofcontents

\section{Introduction}
Combinatorial species arose in the work of André Joyal \cite{Joyal1986foncteurs,8e80a3d617d02edd3db8014be1778638f0c16aaa} as categorification of the theory of generating functions \cite{wilf-gfology}. The idea is as simple as it is fruitful: a convenient way to study a sequence of positive integers $\bsb=(b_0,b_1,b_2,\dots)$ equipped with some combinatorial meaning is to consider them as the coefficients of a formal power series $F_{\bsb}(X)\in \bbQ\llbracket X\rrbracket$ (the \emph{generating series} or \emph{generating function} of $\bsb$), most often in \emph{exponential form}, \ie $F_{\bsb}(X) := \sum_{n\ge 0}\frac{b_n}{n!}X^n$. The properties of $F_{\bsb}(X)$ as an algebraic object reflect onto the combinatorial properties of $\bsb$ (species can be added, multiplied, functionally composed; each of these operations performed on $F_{\bsa},F_{\bsb}$ has meaning in terms of the combinatorial object described by $\bsa,\bsb$), and vice versa, the combinatorial properties of $\bsb$ (for example, the fact that its elements satisfy a certain recurrence relation) reflect on the algebraic properties of its generating series. `Generatingfunctionology' is, among other things, the study of combinatorics inspired by the manipulations of formal power series whose coefficients have combinatorial meaning.

Indeed, crafting a bijective proof to grok numerical identities in terms of bijections between finite sets is acknowledged as \emph{the} fundamental problem in enumerative combinatorics (cf.~for example the introduction of \cite{mendez1991mobius}), and generatingfunctionology is of great help in this respect: for Joyal, a `species of structure' arises as a \emph{categorification} of the notion of generating series.

A species of structure consists of a functor $F : \sfP \to \Set$ having domain the category of finite sets and bijections: instead of a countable sequence of numbers $\{b_n\mid n\ge 0\}$, a countable sequence of sets $\{F[n]\mid n\ge 0\}$; properties of the category of all such functors can now be given combinatorial meaning, combinatorial identities acquire meaning as bijective proofs (=isomorphisms of functors), and operations performed on functors express a possibly complicated object as (monoidal) product of simpler bits (we mention some of these isomorphisms in \autoref{important_comb_ids}). Among Joyal's first applications for the language of species there was a particuarly insighftul proof of Cayley's counting theorem for trees \cite{cayley}, a result which paved the way to a booming development of techniques (propelled by the support of an insider of enumerative combinatorics, and genius, as C.G. Rota) in domains such as representation theory of groups \cite{Chen_1993,0cbbef7b84b3481ef76c8e88a86f80d411b7492f,RAJAN1993197,412e8ae396fc12c5284630b51a24e521afbc5fc7}, the study of set partitions \cite{Bonetti_1992,Joni_1979,mendez1993colored}, Möbius functions \cite{mendez1991mobius,Rota_1964,Senato1997}, graph theory \cite{7a8211434eed2811547338107aa8e1aa26e0ff5f}, up to the exciting field of \emph{combinatorial differential equations} \cite{10.1007/BFb0072518,0cbbef7b84b3481ef76c8e88a86f80d411b7492f,a7260d9aa6cf11af5e8a6e16d6cb29f836d2ff04,0445243bd75f64484c47f7db18f2569031b5e3bd}.

This wealth of applications is by no means limited to the field of enumerative combinatorics; the operation of \emph{plethystic substitution} \cite{Bergeron1987,ddbd4291a067388ac9ba6b092f4e901d9fd249f9,Nava_1985} is recognized as the fundamental building block in the theory of \emph{operads} envisioned by J.P. May \cite{may1972geometry,MR1436914} and natural instances of operadic composition arise in algebraic topology and algebraic geometry \cite{Fresse2009,Getzler1998,loday2012algebraic,obrad2017}, logic and computer science \cite{gambo-joy,gambino2022monoidal,yorgey_thesis} (especially due to their link with multicategory theory \cite{Lambek1969,lambek1989multicategories}), theoretical physics \cite{getzler2009operads,getzler1994operads}, and more.

At about the same time, another application of category theory gained momentum: the idea of interpreting \emph{abstract state machines} inside general categories. The line of research initiated by Arbib--Manes \cite{Arbib1975,Pohl1970}, Goguen \cite{Goguen1975,Goguen1973,Goguen1975a}, Naudé \cite{Naude1977,Naude1979}, and developed in recent times by \cite{vanGlabbeek1999,Hermida2004,Jacobs2006,Venema2006} culminated into Ehrig's monograph \cite{Ehrig} on automata `valued' in an abstract monoidal category $\clK$.

The intuition the reader should have is that an \emph{automaton} is a span
\[\label{automatto}\vxy{
		E & A\otimes X \ar[r]^-s\ar[l]_-d & B
	}\]
whose legs represent respectively a dynamical system (yielding a representation of $A$ over a state space $X$), and a function $s$ whose r\^ole is to give a final state (or output, or answer\dots) to the computation performed by $d$.

The treatment made by \cite{Ehrig} provides a systematic, category--theoretic insight into the transition from determinism to non-determinism, that can be seen as the passage from automata in a monoidal category \cite{Meseguer1975}, to automata in the Kleisli category of an opmonoidal monad \cite{Guitart1980,Jacobs2016} (such as, for example, the probability distribution monads for convex spaces, \cite{Doberkat2007,fritz2015convex,Jacobs2018,Jacobs2010,Mateus1999} or one of its companions --the subdistribution or unnormalized distribution monad).

The category-theoretic content of such an approach to `machines' goes a long way: a tentative chronology follows, but it can only scratch the surface of an immense, often submerged, body of research.
\begin{itemize}
	\item \cite{Adamek1974,adam-trnk:automata} introduced the notion of an \emph{$F$-automaton} in order to abstract even further from the monoidal case the `dynamics' igniting the behaviour of an abstract machine; the progression in abstraction is as follows: from Cartesian machines in a category $\clK$ with finite products, i.e. spans $E \leftarrow A\times E \to B$, one goes to monoidal ones, i.e. spans $E \leftarrow A\otimes E \to B$ valued in a monoidal category $(\clK,\otimes)$; these are the objects of categories $\Mly(A,B)$. Subsequently, one abstracts the action of $A\otimes\firstblank$ on $E$ even further, using a generic endofunctor $F : \clK\to\clK$; this is the category $\Mly(F,B)$.
	\item Only few years prior, extensive work of Betti\hyp{}Kasangian \cite{Betti1981,Betti1982} and Kasangian\hyp{}Rosebrugh \cite{Kasangian1990} pushed for the adoption of ‘profunctorial’ models for automata, capable to pinpoint their behaviour, and their minimization, as a universal property \cite{Goguen1972a,Goguen1975a}.
	\item An insightful idea of Katis, Sabadini and Walters \cite{Katis1997,Katis2010} recognized that categories $\Mly(A,B)$ in a Cartesian category give rise to a composition operation
	      \[\vxy{
			      \Mly(B,C)\times\Mly(A,B) \ar[r] & \Mly(A,C)
		      }\]
	      and thus organize as the hom-categories of a bicategory $\KSW(\clK)$.\footnote{Interestingly enough, KSW category can be seen as a lax analogue of the category of `categories with endofunctor' upon which one builds the \emph{Spanier-Whitehead stabilization} of the category of (pointed) CW-complexes, a staple construction in stable homotopy theory \cite{MR0246294},\cite[Chapter I]{LurieHA}.} The bicategory so obtained can be concisely described as the bicategory of pseudofunctors, lax natural transformations and modifications $\bfB\bbN \to\clK$, where $\bfB\bbN$ is the monoid of natural numbers, regarded as a bicategory with a single object together with $\clK$. This definition extends to monoidal automata in a straightforward way, but there one loses the description as spans, given that the monoidal product isn't universal.
	\item in \cite{guitart1974remarques,Guitart1980} René Guitart introduces the bicategory $\Mac$ as a refinement of a bicategory of spans.\footnote{Note in passing that this is related to Betti, Kasangian, and Rosebrugh idea as two-sided fibrations and profunctors are well-known equivalent ways to present the same bicategory.} In \cite{Guitart1977}, Guitart proves that $\Mac$ is simply the Kleisli bicategory \cite[§4]{fiore2016relative}, \cite{gambino2022monoidal} of the 2-monad of cocompletion under lax colimits. This theme is reprised in \cite{guitart1978bimodules} where Guitart introduces the notion of \emph{lax coend} \cite{hirata2022notes,coend-calcu} as a technical preliminary to expand on the theme of \cite{Guitart1977}.
	\item Building on \cite{Ehrig}, but apparently unaware of \cite{guitart1974remarques}, R. Paré proposed in \cite{Par2010} the notion of a \emph{Mealy morphism} as a proxy between strong functors and profunctors in any $\clV$-enriched category $\clC$. The paper culminates in the impressively general and elegant\footnote{The reader suspecting that this is an overstatement shall rest with the thought that this straightforward statement bestows the bicategory $\clV$-Mly with a clear-cut universal property producing analogues, in one fell swoop, of $\KSW$ and Guitart's categories for \emph{every} suitably cocomplete base of enrichment.} result that the bicategory of $\clV$-Mealy maps is simply the Kleisli bicategory of the lax idempotent 2-monad of $\clV$-copower completion.\footnote{The reader will have noticed a repeating theme: categories that naturally arise organizing computational machines share a universal property of Kleisli type (they are initial in some sense, for ways of factoring a certain monad), and the monad is `of property type', i.e. it is a 2-monad of cocompletion under certain shapes \cite{kock1995monads,Zoberlein1976}.}
	\item In a joint work \cite{EPTCS397.1} the present author explores how KSW's `circuits' and Guitart's $\Mac$ connect via a \emph{local adjunction} \cite{Jay1988,kasangian1983cofibrations}, and can be used to enhance categorical automata into widgets `typed' over a bicategory with possibly more than one object; in short, it allows the passage from a bicategory of automata to \emph{automata in a bicategory}, drawing some ideas from Bainbridge's \cite{Bainbridge1975,Bainbridge1972}. Despite its relative obscurity, likely due to its cutting-edge nature, Bainbridge recognized and made clear the importance of bicategory theory as a foundational language for the theory of abstract automata and, in particular, proposed the idea of left/right Kan extensions along an `input scheme' to analyze behaviour and minimization.
\end{itemize}
Pushing further these ideas, building on all this work, intersects the most prolific branches of modern category theory. To sum up, we find ourselves in the following situation today: a forgotten school of category theorists hid an exciting claim behind a curtain of 2-dimensional algebra:
\begin{quote}
	A piece of \emph{formal category theory} as envisioned by \cite{Gray1974,CTGDC_1975__16_1_17_0,guitart1982logique,street1978yoneda,weber2007yoneda,Weber2016,wood1982abstract} serves as the mathematical foundation of abstract state machines.
\end{quote}
This intriguing hypothesis is scattered across various sources, often unaware of each other; it has been hinted at multiple times and continues to leave traces of its presence for those willing to follow it. We are left with a conjecture and a clear work plan: can this fundamental guiding principle be taken seriously and formalized? Whoever is willing to take up the challenge of verifying this claim is now tasked with lifting the curtain and exploring a rich fauna of categorical widgets.

The present work grafts on top of the wide branches of this overarching project, studying categorical automata theory specialized to the \emph{differential 2-rig} (a notion introduced by the author in \cite{diff2rig}) of Joyal's combinatorial species. In order to do so, it develops further the basic theory of differential 2-rigs, expounded in \cite{diff2rig}: loosely speaking, a differential 2-rig (`for the doctrine of coproducts') consists of a monoidal category $(\clR,\otimes,\de)$ satisfying the following assumptions:
\begin{itemize}
	\item each functor $A\otimes\firstblank$ and $\firstblank\otimes B$ commutes with finite coproducts, to the effect that $\clR$ embodies the structure of a categorified semiring, realized \emph{qua} category equipped with a `bilinear' tensor product;
	\item there exists a functor $\de : \clR\to\clR$ which commutes with coproducts, \ie there is a canonical natural isomorphism
	      \[\notag\de A + \de B\cong \de(A+B)\]
	      and `satisfies the Leibniz rule', \ie there is a natural isomorphism
	      \[\notag\de A\otimes B+A\otimes\de B \cong \de(A\otimes B).\]
\end{itemize}
Starting from this definition, one develops a categorification of differential algebra (intended for example in the sense of \cite{Kolchin1973,Kaplansky1976,Marker2000}), recognizing to the category of Joyal species the same role that in commutative algebra is covered by the ring of polynomials $k[X]$.

On one side, the desire to better understand the theory of differential 2-rigs forces to find examples that motivate general definitions; species constitute such an example: the category $\Spc$ of species is a presheaf topos equipped with a plethora of tightly-knit monoidal structures interacting with a differential structure; this richness implies that when used as an ambient category for monoidal/functorial automata, it gives rise to an interesting theory that, when stated at the correct level of abstraction, is `stable under small perturbations', which means that similar results to the ones presented here export without much effort to presheaf categories equipped with a plethystic substitution operation, such as \emph{coloured} species \cite{mendez1993colored}, \emph{linear} species (both in the sense of \cite{10.1007/BFb0072518} and in the sense of $k\emdash\Mod$-enriched, \cite{aguiar2010monoidal,Getzler1998}), \emph{M\"obius} species \cite{mendez1991mobius}, nominal sets \cite{Pitts2013},\dots{} and it allows to predict what happens when abstract automata are interpreted in a differential 2-rig other than $\Spc$, generalizing \autoref{its_diffe}.

On the other hand, the desire to understand better the theory of combinatorial species forces one to regard the category $\Spc$ as an object of a larger universe of `2-rigs'; a fundamental feature of $\Spc$ is that its derivative endofunctor $\de$ admits both a left and a right adjoint denoted $L$ and $R$. The categories of automata
\[\Mly[\Spc](\de,B),\Mly[\Spc](L,B),\Mly[\Spc](L\de,B),\Mly[\Spc](\de L,B)\]
all constitute interesting examples of categories of automata, yielding a `dynamical system' interpretation for  diagrams of the form $X\leftarrow \de X\to B$, $X\leftarrow L\de X\to B$, etc., at the same time motivating a study of the categories of the endofunctor algebras for such functors (there is to date no reference for such a study).

The present work draws from both the desire to understand $\Spc$ and $\clMlyTens[\Spc]$ as (differential) 2-rigs, and the desire to understand (differential) 2-rigs through specific examples; let's now outline more precisely the structure of the paper.
\subsection{Outline of the paper}
We start introducing the category of species and categories of automata in the sense of \cite{adam-trnk:automata}. The material on species that we need is classical, drawing upon various sources such as \cite{bergeron1998combinatorial,gambo-joy,yorgey_thesis,yeh1985combinatorial}; in \autoref{mly_n_mre} we rework an equally `classical' construction of the categories $\Mly(F,B)$ and $\Mre(F,B)$, drawing from \cite{Ehrig,Guitart1980}, without being afraid of defining objects via universal construction, wherever possible.\footnote{It is in our best interest to state the basic definition at the most formal category-theoretic level fathomable, as our leading intuition about `objects of Mealy and Moore type' is that they are models of a certain $\Cat$-enriched sketch in the sense of \cite{borceux1998theory}, that as such can be valued in any bicategory of choice. \autoref{prob_1} makes this more precise and paves the way for further discussion.} In \autoref{omega_limit}, we introduce the concept of `$\omega$-differential limit', as an intuition for what the terminal object in $\Mly(F,B)$/$\Mre(F,B)$ should represent; the terminology is somewhat borrowed from ergodic theory (specifically, the notion of $\omega$-limit, see \cite[Def. 1.12]{book_91686061}). Later, in \Cref{fib_propes}, we thoroughly explore the fibrational properties of the $\Mly$ construction, yielding the 2-fibration of the \emph{total Mealy 2-category} $\textbf{Mly}$, along with two-sided fibrations \cite{street1980fibrations} $\clMly$/$\clMre$ allowing to consider all dynamics and all outputs at the same time, coherently. In particular, we show that $\clMly$/$\clMre$ are the total categories of a certain opfibration of endofunctor coalgebras, in the sense of \cite{fib_of_alg}, which has the universal property of a coalgebra object (an inserter, \cite{2catlimits}, of the form $\textsf{Ins}(1,R)$ for a certain endofunctor $R$, cf.~\autoref{its_always_a_foa}) In \Cref{mon_tot_mach} we define the \emph{monoidal Mealy fibration} as a particular instance of this construction. The fundamental result of \cite{Katis1997}, defining the KSW category of a monoidal category $(\clK,\otimes)$ arises (\autoref{its_promonad}) when the profunctor associated to the monoidal Mealy two-sided fibration carries the structure of a promonad, of which $\text{KSW}(\clK,\otimes)$ is the Kleisli object. In \autoref{its_prese} we address the issue of lifting accessibility from $\clK$ to $\Mly$/$\Mre$, consolidating the idea that nice properties of the ambient category lift easily to its category of automata.

The first central result of the paper is that assuming $\clK$ is a differential 2-rig in the sense of \cite{diff2rig}, $\clMly$ and $\clMre$ are differential 2-rigs as well: an upshot of \cite{diff2rig} is that differential structures on a category are `difficult to create', and yet categories of $\clK$-valued automata consistute additional examples of differential 2-rigs, simply but not trivially related to $\clK$. Moreover, the adjoints to $\de : \clK\to\clK$ lift to adjoints to $\bar\de$. This paves the way for studying the theory of scopic 2-rigs in two directions: which properties of $(\clK,\de)$ are ensured by the fact that it is scopic, and which constructions on $\clK$ output scopic 2-rigs? We present some results in this direction in \autoref{prop_1}, \ref{scopic_r_fancy}.

The second central result of the paper (not unrelated to the first, given the centrality of endofunctor algebras for the construction of $\clMlyTens$) is \autoref{anchelui_diffrig}, \ref{the_taylor}, where we prove that given a differential 2-rig $(\clK,\otimes,\de)$ the category of $\de$-algebras is a differential 2-rig as well. But then, on the category of $\de$-algebras there is a derivation (acting essentially applying $\de$ to the carrier of an algebra), which in turn has a category of algebras, over which there is a derivation\dots{} This construction yields an opchain of forgetful functors, cf.~\Cref{furetto},
the limit of which we dub $\Jet[\clK,\de]$; this category stands to $\clK$ in a relation that, from a distance, seems vaguely analogous to the relation between a manifold $M$ and its \emph{jet bundle}, \cite[Chapter IV]{Kolar1993}.

Then we turn to the task of studying (Mealy) automata in species, focusing on the particular case where $\clK$ is the category of \Cref{sec_species}; given its structure of differential 2-rig, we are particularly interested in studying \emph{differential dynamics}, i.e. in studying categories $\Mly[\Spc](F,B)$ where the generator $F$ of dynamics is induced by the derivative functor. Given the results in \cite{rajan1993derivatives}, recalled in \autoref{triple_der_adj}, there is plenty of choice for such $F$'s: the triple of adjoints $L\dashv \de\dashv R$ generates four functors, a comonad-monad adjunction $L\de\dashv R\de$ and a monad-comonad adjunction $\de L \dashv \de R$ (paying tribute to the `twelvefold way' of \cite{Stanley2000wc}, we dub the study of this quadruple of pairwise adjoint functors the `fourfold way'); each of these adjunctions generate monads or comonads $R\de L\de$, $L\de R\de$,$\de R\de L$, $\de L\de R$ (and all these are finitely accessible functors because $R$ is). \Cref{other_flavs} extends the results of the paper to other `species-like' categories: coloured (=multisorted) species, linearly ordered species, nominal sets.
\begin{remark}
	The paper in its present form is an extended version of the note published in the proceedings of \textsc{cmcs}\oldstylenums{2024} as \cite{loregian2024automata}; the following results are not present in that paper and should be considered new contributions.

	\begin{itemize}
		\item \Cref{struc_of_algD}, where we study the properties of $\de$-algebras, with particular interest in the possibility of making $\Alg(\de)$ a differential 2-rig on its own (cf.~\autoref{anchelui_diffrig}) and the forgetful functor a (strict) differential 2-rig morphism.
		\item \Cref{lifta_alle_monadi}, where the previous results are adapted to lifting $\de$ to the \textEM categories of $R\de,\de L$ (but the lifting is not to a monoidal category, cf.~\autoref{non_lo_era}).
		\item \Cref{smut_n_scopic}, where we investigate more closer the formal consequences of assuming that the derivative functor $\de$ of a differential 2-rig $\clR$ admits a right adjoint, cf.~\autoref{def_smooth_2rig}, and both a right \emph{and a left} adjoint (respectively, we call such $\clR$'s \emph{right scopic}, \emph{left scopic}, and \emph{scopic}, cf.~\autoref{scopic_2_rig}); we define the `Arbogast algebra' of a differential 2-rig in \autoref{arbogast_algebra}.
		\item \Cref{other_flavs}, where we sketch how to extend the major results proved so far for usual species to various `species-like' categories such as coloured (=multisorted) species, linearly ordered species, nominal sets, etc.
	\end{itemize}
\end{remark}

\section{Combinatorial species}\label{sec_species}
Although we define the category of species from first principles, we refrain from giving a really self-contained presentation of the theory (and more importantly, we cut most ties with combinatorics, relying solely on category theory). For this reason, we advise the reader to consult external resources; although Joyal's original work \cite{8e80a3d617d02edd3db8014be1778638f0c16aaa,Joyal1986foncteurs} remains unparalleled in terms of insight, there are excellent introductory texts and surveys on the category of species in \cite{bergeron1998combinatorial,gambo-joy,yorgey_thesis,yeh1985combinatorial}.

The typical object of the category of species consists of a countable family $\{X_n\mid n\ge 0\}$ of sets, each of which is a (left) $S_n$-set, where $S_n$ denotes the group of permutations of the set with $n$ elements. Choosing degreewise equivariant maps as morphisms, this defines the category $\Spc$ of \emph{combinatorial species}.

After defining $\Spc$ in \autoref{spc_n_Vspc}, we recall how its various monoidal structures interact with each other, with particular attention to the \emph{Day convolution} monoidal structure, cf.~\autoref{various_mono}.\ref{ms_2}, and to the \emph{differential 2-rig} structure, in the sense of \cite{diff2rig} that $\Spc$ carries,\footnote{We recall the definition in \autoref{diff_of_spc}. A differential 2-rig is, roughly, a monoidal category with coproducts, and equipped with a functor $\de$ that satisfies a categorified Leibniz rule, $\de(A\otimes B)\cong \de A\otimes B + A\otimes\de B$.} stressing in particular the fact that the derivative functor $\de : \Spc\to\Spc$ has both a left and a right adjoint; a number of formal consequences follow from this fact, which means that there is a `synthetic theory' of categories that behave like species, categories that in \autoref{scopic_2_rig} we call \emph{scopic}\footnote{From the Proto-Indo-European root \emph{*spe\'{k}-} derived the Latin word \emph{speci\=es} and the Greek verb \emph{skop\'eo}, related to the verb `to see'.} and that we study in \autoref{scopic_r_fancy}, \autoref{prop_1}, \autoref{mangusta}.
\begin{definition}[Species and $\clV$-species]\label{spc_n_Vspc}
	Let $S$ be a set, and $\clV$ a \emph{Bénabou cosmos} (=a symmetric monoidal closed category admitting all limits and colimits, cf.~\cite[p.\@ 1]{Street1974}). Let $\sfP[S]$ denote the free symmetric monoidal category on $S$, regarded as a discrete category.

	The category $(S,\clV)\emdash\Spc$ of \emph{($S$-colored) $\clV$-species} is defined as the category of functors $\fkF : \sfP[S] \to \clV$, and natural transformations.
\end{definition}
All along the paper we will particularly be interested in the category $(1,\Set)\emdash\Spc$ (that we dub simply $\Spc$), where $1$ is a singleton set.\footnote{Other possible choices for $\clV$ are: the category $\Mod_R$ of modules over a ring $R$ (if $R$ is a field, we call a $\Mod_k$-species just a \emph{vector species}, see \cite{loday2012algebraic} for a comprehensive introduction) or the category $\Top_*$ of pointed topological spaces equipped with the smash product \cite[3.6.2]{strom2011modern} (for applications to algebraic topology, see e.g. \cite{markl2007operads}; for a broader notion of operad cf.~the excellent readings \cite{curiennone,Trimbled}). In the majority of references on the subject a $\Mod_k$-species is called a \emph{linear} species; however, in \cite{a0e13a522d655a338dcd3eb0d86f31d83d15c9bd} the term has a different meaning, and the clash of notation might create confusion in the present work.} See however \autoref{other_flavs} for a brief outline of how the results in this paper extend to colored species (\autoref{colo_spc}) and other `species-like' categories, \autoref{sec_linspc}, \ref{sec_ordspc}, \ref{sec_mobspc}.

The theory of combinatorial species can only be understood when one appreciates that $\sfP$ and $\Spc$ enjoy multiple different universal properties at the same time:
\begin{remark}[A number of universal properties for $\sfP$]\label{some_ups_P}\leavevmode
	\begin{enumtag}{ups}
		\item\label{ups_1} The category $\sfP=\sfP[1]$ is the \emph{groupoid of natural numbers}, having as objects the nonnegative integers, and where the set of morphisms $n\to m$ consists of the set of bijections between $[n]$ and $[m]$, if $[n] = \{1,\dots,n\}$ is an $n$-set (so in particular $\Zero$ is the empty set, and each $\sfP(n,m)$ is empty if $n\ne m$). By construction, composition is only defined between endomorphisms, and it coincides with the composition of permutations, as soon as $\sfP(n,n)$ is recognized as the symmetric group $S_n$.
		\item\label{ups_2} Again by construction, the category $\sfP$ is the skeleton of the \emph{groupoid of finite sets} $\Bij$, the category having objects the finite sets $A,B,\dots$ and morphisms $A\to B$ the set of all bijections between $A$ and $B$ (so, in particular, if $A$ and $B$ do not have the same cardinality, $\Bij(A,B)$ is empty).
		\item\label{ups_3} 	Note that $\Bij$ is in turn the core (=the largest subcategory that is a groupoid) of the category $\Fin$ of all finite sets and functions.
		\item\label{ups_4} The (commutative, i.e. strictly symmetric) monoidal structure on $\sfP$ is given by sum of natural numbers, i.e. $[n]\oplus[m]=[n+m]$, the unit is $\Zero$ and permutations act by juxtaposition. (In the non-skeletal model $\Bij$ of the previous point, $A\oplus B$ is the disjoint union of $A$ and $B$, but the monoidal structure is \emph{not} coproducts.)
	\end{enumtag}
\end{remark}
In the following we denote a species in Gothic fraktur as $\fkF,\fkG,\fkH,\dots : \sfP \to\Set$ and call an element $s\in\fkF[n]$ a \emph{species of $\fkF$-structure}.
\begin{corollary}
	Denote $BG$ the group $G$ regarded as a single-object category. The universal property of $\sfP$ entails that there is an isomorphism of categories $\sfP\cong\sum_n BS_n$ where the right-hand side is the coproduct in the category of groupoids; as a consequence $\Spc\cong \prod_n\Set^{S_n}$ where each $\Set^{S_n}$ is the category of left $S_n$-set:
	\[\textstyle\Cat(\sfP,\Set)\cong\Cat\big(\sum_{n\ge 0} BS_n,\Set\big)\cong \prod_{n\ge 0}\Cat\big(BS_n,\Set\big).\label{repre_repre}\]
	As a consequence, species can equivalently be presented as a \emph{symmetric sequence} $\{X_n\mid n\ge 0\}$ of sets, each of which is equipped with a (left) $S_n$-action $S_n\times X_n \to X_n$.
\end{corollary}
\begin{definition}[Change of base for species]\label{ch_o_bases}
	Let $\clV$ be a monoidal category monadic over $\Set$ via a functor $K : \clV\to\Set$ which is lax monoidal (for example the forgetful functor $U : \Mod_R\to\Set$); then there is a \emph{base change} adjunction $F_* : \Spc \rightleftarrows \clV\emdash\Spc : U_*$ induced through the free-forgetful adjunction $F\dashv U$. For example, if $F : \Set \to\Mod_k$ is the free $k$-vector space functor, we denote $k\langle \fkH\rangle$ the vector species $F_*\fkH$ induced by a $\Set$-species $\fkH$: we consider the vector space having the species of $\fkH$-structure as basis vectors.
\end{definition}
\begin{example}[Some important species]\label{exam_species}
	In various parts of the present work we will consider the following species.
	Many more examples can be found in \cite[Ch. 1]{bergeron1998combinatorial}.
	\begin{enumtag}{es}
		\item\label{es_0} Given an object $V$ of $\clV$, there is a unique symmetric monoidal $\clV$-species $c_V$ sending $[n]$ to $V^{\otimes n}$. If $V=I$ is the monoidal unit, $c_I$ is called the `exponential species' $\Exp$. The exponential $\Set$-species is just the constant functor at the terminal object.\footnote{In a serendipitous choice, the notation $\Exp$ for this species hints at the same time that $\Exp$ is the species of sets, or \emph{éspèce des ensembles}, and that it's an analogue of the exponential function, as $\de \Exp[n] = \Exp[n]$, for the derivative functor of \autoref{diff_of_spc}.}
		\item\label{es_1} The species $\wp$ of \emph{subsets} sends an $n$-set $A$ to the $2^n$-set of all its subsets; a permutation acts in an obvious way, since a bijection $\sigma : A\to A$ induces a bijection $\sigma^* : 2^A\to 2^A$ by functoriality.
		\item\label{es_2} The species $\Lin$ of \emph{total orders} 
		sends $[n]$ to the set of total orders on $[n]$, identified with the \emph{set} $|S_n|$ of bijections of $[n]$, over which $S_n$ acts by left multiplication.
		\item\label{es_3} The species $\fkS$ of permutations sends each finite set $[n]$ into the (carrier of the) symmetric group on $n$ letters, $S_n$. The symmetric group acts on itself by conjugation: if $\tau \in S_n$, $\sigma : S_n \to S_n$ is the map sending $\tau\mapsto \sigma\tau\sigma^{-1}$.
		\item\label{es_4} The species $\Cyc$ of \emph{oriented cycles} sends a finite set $[n]$ to the set of inequivalent (i.e. not related by a cyclic permutation) ways to sit $n$ people at a round table, or more formally, in the set of \emph{cylic orderings} of $\{x_1,\dots,x_n\}$. As $\Cyc[n]$ identifies with the set of cosets $S_n/C_n$ ($C_n$ the cyclic group), one derived that $|\Cyc[n]| = (n-1)!$.
	\end{enumtag}
\end{example}
The category of species exhibits a fairly rich structure that we now schematically review.\footnote{An important additional universal property we do not need in our analysis is that $\Spc$ is a Grothendieck topos, precisely the classifying topos \cite[Ch. VIII]{mac1992sheaves} for $\clP$-torsors, where $\clP$ is the category $\sfP$ regarded as a groupoid.}
\begin{proposition}\label{up_spc_1}
	$\Spc$ is the free cocompletion $\widehat{\sfP}$ under small colimits \cite[Remark 2.29]{Ulmer1968} of $\sfP^\op$ (which is isomorphic to $\sfP$, being a groupoid); as such, for every cocomplete category $\clD$ there is an equivalence of categories
	\[\Cat(\sfP,\clD)\cong \{\text{colimit preserving functors } \Spc \to \clD\}\]
	given by `Yoneda extension' \cite[Ch. 2]{coend-calcu}.
\end{proposition}
\begin{proposition}\label{up_spc_2}
	Following \cite{adamek2008analytic,Hasegawa2002,Joyal1986foncteurs} $\Spc$ is the (nonfull) subcategory of \emph{analytic endofunctors} of $\Set$, i.e. those endofunctors $F : \Set\to\Set$ such that, if $J : \Bij\to\Set$ is the tautological functor $[n]\mapsto [n]$, the left Kan extension of $FJ$ along $J$ coincides with $F$. The usual coend formula \cite[X.5,6]{working-categories} to express $\Lan_JFJ$ entails that $F$ is analytic if and only if it acts on a set $X$ as
	\[FX\cong \int^n F[n] \times X^n\]
	i.e. if and only if $F$ admits a `Taylor expansion' $\sum_{n=0}^\infty F[n]\frac{X^n}{n!}$; hence the name. The series $g_F(X)=\sum_{n=0}^\infty |F[n]|\frac{X^n}{n!}\in\bbQ\llbracket X\rrbracket$, where $|S|$ denotes the cardinality of a set $S$, is called the (exponential) \emph{generating series} \cite[§1.2]{bergeron1998combinatorial} of the species $F$.
\end{proposition}
In the following statement, a (cocomplete) \emph{2-rig} consists of a monoidally cocomplete category $(\clR,\otimes,I)$: each tensor functor $A\otimes\firstblank, \firstblank\otimes B$ commutes with all small colimits. More generally (we will introduce the notion in \autoref{doctrines_for_all}) a $\bfD$-cocomplete 2-rig, or $\bfD$-2-rig for short, is such that each $A\otimes\firstblank, \firstblank\otimes B$ commutes with colimits of a certain shape $\bfD$.
\begin{proposition}[\protect{\cite[§5]{diff2rig}}]\label{up_spc_3}
	$\Spc$ is the free cocomplete \emph{2-rig} (as in \cite[§5]{diff2rig}) on a singleton; as such, given a cocomplete 2-rig $\clR$ there is an equivalence of categories
	\[ \clR \cong \{\text{colimit preserving 2-rig functors } \Spc \to \clR\}.\]
	In \cite[§5]{diff2rig} we observe how to construct the free cocomplete (symmetric) 2-rig on a given category $\clA$ it suffices to take the free (symmetric) monoidal category on $\clA$ (call it $\sfP[\clA]$, concordant with \autoref{spc_n_Vspc}), and subsequently, its free cocompletion $\widehat{\sfP[\clA]}$. In [\emph{ibi}] the notion of morphism of 2-rigs is given indirectly as pseudomorphism for the particular `doctrine of $\bfD$-rigs' in study.
\end{proposition}
This last characterization requires a more fine-grained analysis of the various monoidal structures $\Spc$ can be equipped with.
\begin{remark}\label{various_mono}
	The category $\Spc$ of species carries
	\begin{enumtag}{ms}
		\item\label{ms_1} the \emph{Cartesian} (or \emph{Hadamard} \cite[8.1.2]{aguiar2010monoidal}) monoidal structure, the product of species being taken pointwise; the monoidal unit for the Hadamard product is the species that is constant at the singleton. Dually, the \emph{coCartesian} monoidal structure, the coproduct of species being taken pointwise (together with the structure above, $\Spc$ is $\times$-distributive and forms a `biCartesian closed' category in the sense of \cite{1978145}); however, its biCartesian structure is not very interesting, compared to
		\item\label{ms_2} the \emph{Day convolution} (or \emph{Cauchy} \cite[8.1.2]{aguiar2010monoidal}) monoidal structure, given by the universal property of $\Spc$ as the free monoidally cocomplete category on $\sfP$ \cite{imkelly} as the coend
		\[\label{day_def}(F \Day G)[p] := \int^{mn} F[m]\times G[n] \times \sfP(m+n,p)\]
		(Note in passing that the $\Day$-monoidal structure is symmetric and closed with an internal hom $\homDay$.) In particular, $\sfP$ is monoidally equivalent to the subcategory of $\Spc$ spanned by representables, and thus the $\Day$-monoidal unit is $y\Zero$.
		\item\label{ms_3} the \emph{substitution} (or \emph{plethystic}, cf.~\cite{mendez1993colored,Nava_1985}) monoidal structure, defined for $F,G : \sfP\to\Set$ as
		\[(F\circ G)[p] = \int^n Fk\times G^{\Day k}[p],\]
		where $G^{\Day k}:=G\Day G\Day\dots\Day G$ is the Day convolution iterated $k$ times. The $\circ$-monoidal unit is the representable $y\Uno$.
		Note in passing that the $\circ$-monoidal structure is \emph{not} symmetric, and only \emph{right} closed, i.e. only $\firstblank\circ G$ has a right adjoint.
	\end{enumtag}
\end{remark}
All these monoidal structures are tightly related:
\begin{remark}
	The Hadamard and Day convolution product give $\Spc$ the structure of a \emph{duoidal category} in the sense of \cite{garner2016commutativity}: $(\Spc,\times,\Day)$ and $(\Spc,\Day,\times)$ \cite[8.13.5]{aguiar2010monoidal} are both duoidal; \emph{positive} species, i.e. those for which $F[\varnothing]=\varnothing$ form a duoidal category under substitution and Hadamard product, [\emph{ibi}, B.6.1]. All these results extend to $\clV$-species using its monoidally cocomplete structure. The plethystic structure makes $\Spc$ monoidally equivalent to the category of analytic functors under composition \cite{adamek2008analytic,Joyal1986foncteurs}.
\end{remark}
\begin{remark}\label{important_comb_ids}
	As an additional demonstration of how tightly the Hadamard, Cauchy and plethystic structures are related, we record how these identification between combinatorial species hold \cite{bergeron1998combinatorial}: simple species tensored togeter can build quite sophisticated objects, as
	\begin{enumtag}{ci}
		\item the species of subsets $\wp : A\mapsto 2^A$ is isomorphic to $\Exp\Day \Exp$;
		\item the species $\fkS$ of permutations of \autoref{exam_species}.\ref{es_3} is isomorphic to the substitution $\Exp\circ\Cyc$;
		\item more generally, for every species $\fkF$ the species of structure the substitution $\Exp\circ \fkF$, applied to a finite set $A$, consist of an $r$-partition $(U_1,\dots,U_r)$ of $A$, and a species $s_i$ of $\fkF$-structure on each $U_i$.
	\end{enumtag}
\end{remark}
We record in the form of a definition/proposition some useful adjunctions involving the category of species.
\begin{definition}[Levels of the topos of species]
	For every $n\ge 0$, we denote $\iota_n : \sfP_{\le n} \hookrightarrow \sfP$ the inclusion of the full subcategory of $\sfP$ on the objects $\{\Zero,\Uno,\dots, [n]\}$.

	Precomposition with $\iota_n$ determines a \emph{truncation} functor
	\[\iota^*_n := \tau_n : \Spc = [\sfP, \Set]\to [\sfP_{\le n}, \Set] \]
	and left and right Kan extensions along $\iota_n$ put $\tau_n$ in the middle of a triple of adjoint functors $l_n \dashv \tau_n \dashv r_n$.

	Given $\fkH \in [\sfP_{\le n}, \Set]$, the species $l_n \fkH$ (resp., $r_n \fkH$) can be characterized as the left (resp., right) Kan extension of $\fkH$ along $\iota_{\le n}$; one easily sees that these functors can be described explicitly as
	\[
		l_n\fkH : m \mapsto \begin{cases} \fkH [m] & m \le n \\ \varnothing & m > n \end{cases} \qquad \qquad r_n\fkF : m \mapsto \begin{cases} \fkF[m] & m \le n \\ * & m > n \end{cases}
	\]
	We will say that a species $F$ has a \emph{contact of order $n$} with a species $G$ if $\tau_n F = \tau_n G$. We denote this relation as $F \mathrel{\sim_n} G$.

	It is clear that $F$ has contact of order $n$ with $G$ if and only if the associated series in the sense of \autoref{up_spc_2} define the same element in the quotient $\bbQ\llbracket X\rrbracket/(X^{n+1})$.
\end{definition}
\begin{definition}[Convergence]
	A \emph{sequence} of species is an ordered family of species $(F_0,F_1,\dots)$. The sequence $(F_0,F_1,\dots)$ is said to \emph{converge} to the species $F_\infty$ if the following `Cauchy' condition is satisfied:
	\begin{quote}
		For every $N\ge 0$ there exists an index $\bar n$ such that for every $n\ge \bar n$, $F_n \mathrel{\sim_N} F_\infty$.
	\end{quote}

	In simple terms, $(F_0,F_1,\dots)$ converges to $F_\infty$ if for every $N\ge 0$, all but a finite initial segment of terms of the sequence have contact of order $N$ with $F_\infty$. If this is the case, we write $F_n \overset{n\to\infty}\rightharpoondown F_\infty$.
\end{definition}
\begin{definition}[Species as structured graded sets]
	Every symmetric group $S_n$ admits a terminal morphism into the trivial group; precomposition along this terminal map defines a functor $j_n : \Set^{S_n} \to \Set$
	yielding the carrier of a (left) $S_n$-left $X$; left and right Kan extension along $j_n$ then define a triple of adjoints
	\[\vxy[@C=2cm]{U_n\dashv j_n \dashv V_n : \Set^{S_n} \ar[r]|-{j_n} & \Set\pair[.75em]{V_n}{U_n}[l]}\]
	and the 2-functoriality of the product $\prod_{n\ge 0}$, together with the chain of isomorphisms in \Cref{repre_repre}, yield that there is a triple of adjoints
	\[\label{innocent_monadic}\vxy{U\dashv j\dashv V : \Spc \ar[r]|-{j} & \pair[.75em]{V}{U}[l] \Set^\bbN.}\]
	(Note, in passing, that $j$ is monadic: this yields a characterization of $\Spc$ as the category of $jU$-algebras, such an algebra being specified exactly equipping an object $(X_n)\in\Set^\bbN$ with the structure of a symmetric sequence.)
\end{definition}
\subsection{Species as a differential 2-rig}
An important feature of $\Spc$ that we will analyze in this paper is that it is a \emph{differential 2-rig}: the notion was introduced by the author in \cite{diff2rig} as a unifying language capturing instances of monoidal categories $(\clR,\otimes,I)$ where each endofunctor $A\otimes\firstblank,\firstblank\otimes B$ of tensoring by a fixed object is cocontinuous and there is an endofunctor $\de : \clR \to\clR$ that is `linear and Leibniz' in the sense specified by \cite[§4]{diff2rig} and \autoref{def_gen_diff2rig} right below.

Observe that, as it is true in all presheaf categories equipped with Day convolution, the tensor product $\Day$ of \autoref{various_mono}.\ref{ms_2} preserves colimits separately in each variable (i.e., each $A\Day\firstblank$ and $\firstblank\Day B$ is cocontinuous); moreover,
\begin{remark}[The differential structure of $\Spc$]\label{diff_of_spc}
	The category $\Spc$ of species is equipped with a `derivative' endofunctor $\de : \Spc\to\Spc$ (cf.~\cite[§1.4 and \emph{passim}]{bergeron1998combinatorial}) defined as $\de F : [n]\mapsto F[[n]\oplus \Uno]$, or in the non-skeletal model $\Bij$ sending a finite set $A$ to the set $\fkF[A^+]$, where $A^+ := A\sqcup 1$ is the set $A$ to which a distinguished point has been adjoined.

	Such functor satisfies the following properties:
	\begin{enumtag}{d}
		\item \label{d_1} $\de$ is `linear', i.e. it preserves all colimits (in particular, coproducts, hence the linearity of the derivative operator: $\de(\fkF+\fkG)\cong \de\fkF + \de\fkG$);
		\item \label{d_2} $\de$ is `Leibniz', i.e. it is equipped with tensorial strengths $\tau' : \de A\otimes B \to \de(A\otimes B)$ and $\tau'' : A \otimes\de B \to \de(A\otimes B)$ such that the unique map induced by $\tau',\tau''$ from the coproduct of their domains is invertible, to the effect that $\de$ `satisfies the Leibniz rule'
		\[\de A\otimes B + A \otimes\de B\cong\de(A\otimes B).\]
	\end{enumtag}
\end{remark}
\begin{definition}\label{def_gen_diff2rig}
	Every monoidal category $(\clK,\otimes)$ equipped with an endofunctor $\de$ that satisfies the same two properties \ref{d_1},\ref{d_2} is called a \emph{differential 2-rig} (for the doctrine of all colimits) in \cite{diff2rig}.
\end{definition}
\begin{remark}\label{doctrines_for_all}
	With reference to the use of the word `doctrine' in the previous discussion, see \cite[§2]{diff2rig}, where we adapt a concept of \cite{adamek2002classification}, in which a doctrine $\bfD$ is just a subcategory of $\Cat$ (or by extension, the KZ-monad $T_\bfD$ \cite{Zoberlein1976} of free cocompletion under colimits of all shapes $D\in\bfD$); in particular in \cite[2.2---2.5]{diff2rig} we employ the following terminology:
	\begin{itemize}
		\item An \emph{additive doctrine} is a 2-category $\bfD\subseteq\bfCat$ (the 2-category of categories, strict functors, natural transformations) whose objects are categories that admit all colimits of diagrams belonging to a prescribed class, including at least finite discrete diagrams --whose colimits are finite coproducts, denoted with the infix +, and $\varnothing$ for the empty coproduct.
		\item A \emph{multiplicative doctrine} is a 2-category that is monadic (in the 2-categorical sense, \cite[p. 36]{blackwell_thesis}) over the 2-category $\mathbf{MCat}_s$ of monoidal categories, strong monoidal functors, and monoidal natural transformations.
	\end{itemize}
	Intuitively, a multiplicative doctrine consists of a 2-category of monoidal categories, possibly equipped with additional structure, that arises as the category of (pseudo)algebras for a 2-monad $M$ on $\bfCat$. So, a multiplicative doctrine is given by a 2-monad $M$ on $\Cat$, modelled over the 2-monad whose algebras are monoidal categories, of which we consider the 2-category of (pseudo)algebras.

	If by `category admitting $\bfD$-colimits' we understand a $T_\bfD$-pseudoalgebra, a \emph{$\bfD$-2-rig} consists of a monoidal category $(\clK,\otimes)$ admitting $\bfD$-colimits, where all $A\otimes\firstblank$ and $\firstblank\otimes B$ are (strong) $T_\bfD$-pseudoalgebra morphisms. This defines a 2-category $\cate{2Rig}_{\bfD}$ of $\bfD$-2-rigs, their morphisms $F : \clR\to\clS$, and 2-cells (monoidal natural transformations), that however we will rarely invoke in the discussion, content to study 1-dimensional properties. Similary, there is a 2-category $\de\cate{2Rig}_{\bfD}$ of differential 2-rigs, morphisms of the underlying 2-rigs equipped with an isomorphism $F\de\cong \de F$, and monoidal natural transformations.

	Given an additive doctrine $T_\bfD$ and a multiplicative doctrine $M$ as above, then, the \emph{doctrine of $\bfD$-2-rigs} is the (category of pseudoalgebras for the) 2-monad $T = T_\bfD\circ M$ that one obtains from a distributive law $\lambda : M\circ T_\bfD\To T_\bfD\circ M$.

	We will particularly be interested in two special cases of this construction, at extremal levels of generality: on one hand, the minimal assumption for a 2-rig is to be a monoidal category $\clR$, with finite coproducts preserved by ech $A\otimes-$ (these are the \emph{2-rigs for the doctrine of coproducts}); on the other hand, the strongest requirement is that $\clR$ is cocomplete and each $A\otimes -$ is cocontinuous (these are the \emph{2-rigs for the doctrine of all colimits}).
\end{remark}
\begin{remark}\label{on_species_differential_calc}
	In the case of species, the proof that $\de(F\Day G)\cong \de F\Day G + F \Day \de G$ appears in Joyal's original papers introducing combinatorial species. Moreover, it was known to Joyal that $\de$ satisfies the `chain rule', which means that there is a canonical isomorphism
	\[\de(F\circ G)\cong (\de F\circ G)\Day \de G;\]
	cf.~\cite[Theorem 5.18]{diff2rig} for a conceptual proof of this latter result (extendable to other species-like categories). This doesn't happen by accident: to a very large extent, the combinatorial differential calculus of species agrees with the differential calculus that can be done in the ring $\bbQ\llbracket X\rrbracket$ of formal power series (with rational coefficients), under the formal derivative operation $\frac d{dX}(a_0 + a_1X + \frac{a_2}2 X^2+\dots) := a_1 + a_2 X + \frac{a_3}2 X^2+\dots$. In particular, observe that $g_{\de F}(X)$ is the formal derivative $\frac d{dX}g_F(X)$ of the series in \autoref{up_spc_2}.
\end{remark}
\begin{remark}\label{flesh_deriv}
	Part of the fairly rich structure enjoyed by the differential 2-rig $(\Spc,\otimes,\de)$ can be explained with the fact that $\de$ preserves all \emph{limits} as well: $\de$ is precomposition with the functor $\firstblank\oplus\Uno$ defined starting from the monoidal structure in \autoref{some_ups_P}.\ref{ups_4}; but then, call $\Delta=\firstblank\oplus\Uno$, the left (resp., right) adjoint to $\de$ is the left (resp., right) Kan extension along $\Delta$, which exists since $\Spc$ is a presheaf category.\footnote{\label{futtanota}An alternative proof of the same fact is in terms of the Day convolution structure: one sees that there is a natural isomorphism $\de F \cong\homDay[{y\Uno}][F]$ where $\homDay$ is the internal hom, and $y\Uno = \sfP(1,-)$ the corepresentable functor on $\Uno$; now, certainly $\homDay[{y\Uno}]$ must have $y\Uno\Day -$ as left adjoint, but since in every presheaf category representables are tiny objects, $\de$ must also be cocontinuous, hence a left adjoint by the special adjoint functor theorem.}
\end{remark}
We just proved the following result:
\begin{theorem}\label{triple_der_adj}
	The derivative functor $\de : \Spc \to\Spc$ fits in a triple of adjoints $L\dashv \de\dashv R$, and $L,R$ are obtained as Kan extensions ($L$ on the left, $R$ on the right) along the functor $\Delta=\firstblank\oplus\Uno$.
\end{theorem}
This fact was first observed in \cite{rajan1993derivatives}, where the explicit descriptions
\[\label{descr_di_LR} \textstyle
	L\fkF : A\mapsto \sum_{a\in A} \fkF[A\setminus\{a\}]\qquad
	R\fkF : A\mapsto \prod_{a\in A} \fkF[A\setminus\{a\}]
\]
for how $L,R$ act are given in terms of $\fkF$ as a functor $\Bij\to\Set$, and some useful combinatorial identities expressing $L\de, R\de$, $\de L, \de R$ in simpler terms are also analyzed (we recall them in \autoref{lumaca}).
\begin{notation}[Scopic 2-rig]\label{scopic_2_rig}
	We call \emph{scopic 2-rig} a differential 2-rig $(\clR,\otimes,D)$ whose derivative functor $D$ has both a left and a right adjoint, denoted respectively $L$ and $R$.
\end{notation}
\subsection{Left- and right-scopic 2-rigs}\label{smut_n_scopic}
Scopic differential 2-rigs can be understood as particularly well-behaved differential 2-rigs, since a fairly rich set of consequences can be derived from the existence of the string of adjoints $L\dashv D\dashv R$; a number of results can already be deduced for purely formal reasons from the fact that the derivative functor admits an adjoint on only one side. At a purely informal level, the situation should be thought in analogy with the wealth of global properties that follow from assuming that the terminal geometric morphism $\Gamma : \clE\to\Set : s$ of a sheaf topos admits more and more adjoints, and more and more well-behaved, in Lawvere's `axiomatic cohesion' framework, \cite{LAWVERE1994,lawvere07cohesion,538ea74b06cef7ce61f2a4655af64d1abd38dcde,28b28d818c32bd49e0fe548630beb90b8112e179}.

Derivation functors for the doctrine of cocomplete 2-rigs are required to preserve colimits, so that an easy, and already interesting intermediate notion of `smoothness' for a 2-rig is found when $\de : \clR\to\clR$ satisfies the solution set condition; a sufficient condition so that this happens, often realised in practice (and realized in the category of species), is that 
$\clR$ is a presentable category.
\begin{definition}[Left and right scopic differential 2-rig]\label{def_smooth_2rig}
	We call a differential 2-rig $(\clR,\otimes,\de)$ \emph{left} (resp. \emph{right}) \emph{scopic} if the derivative functor $\de : \clR\to\clR$ has a left (resp., right) adjoint $L$ (resp., $R$).
\end{definition}
In \autoref{prop_1} and \ref{scopic_r_fancy} below we understand that a \emph{closed} (differential) 2-rig is a (differential) 2-rig whose underlying monoidal category $(\clR,\otimes)$ is closed; in this case, denote the internal hom as $[-,-]$. Then,
\begin{proposition}\label{prop_1}
	Let $(\clR,\de)$ be a closed, right scopic differential 2-rig; then, the right adjoint $R$ to the derivative functor satisfies a dual condition to the Leibniz property, in the form of a canonical isomorphism
	\[\label{anti_derivative}\vxy{\varsigma_{AB} : [A,RB] \ar[r]^-\cong & [\de A,B]\times R[A,B]}\]
	natural in both objects $A,B\in\clR$.
\end{proposition}
\begin{proof}
	The proof consists of an application of Yoneda, in the same fashion of \cite[2.2.2]{GRAY1980127}: for a generic object $X$ we compute
	\[\clR(\de(X\otimes A),B)\cong\clR(X\otimes A,RB)\cong\clR(X, [A,RB])\]
	On the other hand,
	\begin{align*}
		\clR(\de(X\otimes A), B) & \cong \clR(\de X\otimes A + X \otimes \de A, B)             \\
		                         & \cong \clR(\de X\otimes A, B) \times \clR(X\otimes\de A, B) \\
		                         & \cong \clR(X, R[A,B])\times \clR(X, [\de A,B])              \\
		                         & \cong \clR(X, R[A,B]\times [\de A,B])
	\end{align*}
	and now given that the object $X$ is generic, by Yoneda lemma, we conclude that there must be an isomorphism like \Cref{anti_derivative}.
\end{proof}
Left scopic 2-rigs are an interesting object to study in relation to the possibility of defining a well-behaved category of differential operators, in analogy with a classical result of differential algebra: given a differential ring $(R,d)$, the set of derivations $d : R\to R$ has the structure of an $R$-module, as
\[rd(x\cdot y) = r(dx\cdot y+x\cdot dy) = rdx\cdot y + x\cdot rdy.\]
We can easily find an analogue of such result and prove that the derivations of a differential 2-rig $(\clR,\de)$ form an $\clR$-module.
\begin{definition}\label{additiv_funs}
	The 2-rig $\Cat_+(\clR,\clR)$ is defined as the category of endofunctors $L$ of $\clR$ that commute with finite sums. This is a (strict) 2-rig, since the preservation of sums makes composition of elements bilinear (and not only left linear).
\end{definition}
\begin{definition}\label{subrig_of_stron}
	The sub-2-rig $\Cat_{+,\tau}(\clR,\clR)$ is defined as the full subcategory of $\Cat_+(\clR,\clR)$ spanned by the objects $L : \clR\to\clR$ that (commute with sums and) are equipped with an invertible natural transformation
	\[\label{eq_stro} \vxy{LA\otimes B \ar[r]^-\cong & L(A\otimes B)}\]
	natural in $A,B\in\clR_0$.
\end{definition}
\begin{remark}\label{tensor_hom_der}
	Observe that $\Cat_{+,\tau}(\clR,\clR)$ is a sub-2-rig of $\Cat_+(\clR,\clR)$, equivalent to $\clR$ as \Cref{eq_stro} entails that $LX\cong L(I\otimes X)\cong LI\otimes X$, so that $L$ is isomorphic to the tensor-by-$LI$ functor $LI\otimes\firstblank$. We call a left scopic differential 2-rig $(\clR,\de)$ where $L=LI\otimes\firstblank$ (and thus $\de = [LI,\firstblank]$) `equipped with a tensor-hom derivative'.
\end{remark}
\begin{definition}\label{ders_on_R}
	Let $\clR$ be a differential 2-rig; denote $\Der[\clR]$ the category of differential operators on $\clR$, \ie the category of endofunctors $D : \clR\to\clR$ that are linear and Leibniz in the sense of \autoref{def_gen_diff2rig}; a morphism of derivations is then just a natural transformation of strong endofunctors.
\end{definition}
Note that $\Der[\clR]$ is a full subcategory of the 2-rig $\Cat_+(\clR,\clR)$, and that
\begin{proposition}\label{itsa_module}
	$\Der[\clR]$ is a left $\clR$-module in the sense of \cite{janelidze2001note}; indeed, let $\de\in\Der[\clR]$ and $L\in \clR$, one can easily see that
	\begin{align*}
		L\de(A\otimes B) & \cong L(\de A\otimes B + A\otimes \de B)    \\
		                 & \cong L(\de A\otimes B) + L(A\otimes \de B) \\
		                 & \cong L\de A\otimes B + A \otimes L\de B
	\end{align*}

\end{proposition}
Thus, one can define differential operators out of the derivative symbol $\de$ and tensoring by objects of $\Cat_{+,\tau}(\clR,\clR)\cong \clR$, by closing under arbitrary sums.
\begin{definition}[Arbogast algebra of $\clR$]\label{arbogast_algebra}
	The \emph{Arbogast algebra}\footnote{The French mathematician Louis François Antoine Arbogast ($\ast$1759--$\dag$1803) first introduced in \cite{arbogast1800} the notation `$Df$' to denote the action of a differential operator $D$ on a function $f$, thus paving the way to the notion of a differential operator as a higher-order function $D$ between functional spaces.} of $\clR$, denoted $\Arb[\clR,\de]$, consists of the 2-rig \emph{generated} by the element $\de$ in the 2-rig $\Cat_+(\clR,\clR)$ of \autoref{additiv_funs}.
\end{definition}
\begin{remark}\label{mangusta}
	It is in general quite difficult to determine the structure of $\Der[\clR]$. In the category of species, one finds as a corollary of \autoref{itsa_module} above that the composite functor $L\de$ is itself a derivation (for the doctrine of cocomplete 2-rigs): for species $\fkX,\fkY$
	\begin{align*}
		L\de(\fkX\Day\fkY) & =     y\Uno\Day\de(\fkX\Day\fkY)                                \\
		                   & \cong y\Uno\Day(\de \fkX\Day \fkY) + y\Uno\Day(\fkX\Day\de\fkY) \\
		                   & \cong L\de\fkX\Day\fkY + \fkX\Day LD\de\fkY
	\end{align*}
	In other words, we have $L(\fkX\Day\fkY)\cong L\fkX\Day \fkY$ or, which is equivalent, $\homDay[L\fkX][\fkY]\cong\homDay[\fkX][\de \fkY]$.
\end{remark}
Coupled with the aforementioned \autoref{itsa_module}, this suggest a sufficient condition, in a closed left scopic 2-rig, to make $L\de$ a derivation: indeed, in a closed scopic differential 2-rig one has that the following two conditions are equivalent. 
\begin{enumtag}{ld}
	\item \label{ld_1} $L(\firstblank\otimes \firstblank)\cong L\firstblank\otimes \firstblank$;
	\item \label{ld_2} $[L\firstblank,\firstblank]\cong [\firstblank,\de \firstblank]$ (\ie, if the adjunction $L\dashv\de$ lifts to an enriched adjunction).
\end{enumtag}
\begin{proposition}\label{scopic_r_fancy}
	Let $(\clR,\de)$ be a closed left scopic 2-rig; consider the adjunction $L\dashv \de$, and assume that either of the equivalent conditions \ref{ld_1} or \ref{ld_2} above is satisfied; then the composite comonad $L\de$ is itself a derivation.
\end{proposition}
\autoref{on_species_differential_calc} above and this paragraph pave the way to the possibility of studying `categorified differential algebra', a starting question being, what good properties of $(\clR,\de)$ give nice properties to its Arbogast algebra?

In \autoref{sec_diffeq} below we sketch a proposal for a notion of `differential polynomial' that might be of interest, in mild analogy with \cite{gambino_kock_2013}'s theory of polynomial functors. We postpone a comprehensive discussion of the matter, as this would led us astray from the current goals, but see \autoref{prob_2} for some intuition on what we might focus on in the future.
\subsubsection{Algebraic structures and co/algebras in $\Spc$.}\label{some_coalgebras}
We end the section reviewing the characterization of monoids, comonoids and Hopf monoids in $\Spc$. This will be essential in \autoref{sec_abs_auto_in_spc}, since monoidal automata theory in a category with countable sums forces us to understand the structure of the subcategory of $\Day$-co/monoids at a fundamental level. (Although it is just in the setting of vector species that the notion of \emph{bialgebra} and \emph{Hopf object} becomes of particular importance, as widely exposed in \cite{aguiar2010monoidal,Aguiar2020,Aguiar2022}).

Among other examples, we will consider in \autoref{lin_algebras} categories arising as pullbacks of the forgetful functor $\Spc^\Lin\to\Spc$ from the Eilenberg--Moore category of a monad $\Lin\Day\firstblank$.

First of all, Hadamard co/monoids are simply co/monoid-valued species, i.e. functors $F : \sfP \to\Mon$ or $\sfP\to\Comon$ into the categories of monoids and comonoids in $\Set$ (more generally, a model in $\Spc$ for a certain algebraic theory $\bbT$ is just a species valued in $\bbT$-models, \ie a $(1,\Set^\bbT)$-species; with some care, this result extends to $\clV$-species, a Hadamard monoid with respect to the monoidal product in $\clV$ being just a functor $\sfP \to \Mon(\clV)$).

\emph{Cauchy co/monoids} (i.e. co/monoids for the Day convolution, whence our preference for calling them Day co/monoids) are far more interesting, as well as substitution co/monoids (the latter are called \emph{co/operads} and have an extremely long history, excellent surveys geared towards the different areas of Mathematics using them are \cite{curiennone,gambo-joy,Kelly2005a,markl2007operads}).
The first remark on $\Day$-co/monoids is simply that there aren't any among representables.
\begin{remark}
	There are no nontrivial representable $\Day$-magmas, for the simple reason that the subcategory spanned by representables is monoidally equivalent to $(\sfP,\oplus)$, and in the latter a binary operation $[n]\oplus[n]=[2n]\to[n]$ can exist only if $2n=n$. For a similar reason, there are no nontrivial `$k$-coary cooperations' $[n] \to [n]^{\oplus k}$.
\end{remark}
\begin{remark}\label{how_a_monoid_in_spc}
	It is worth to explicitly spell out what a $\Day$-monoid $(M,\mu,\eta)$ in $\Spc$ must be made of:
	\begin{itemize}
		\item the unit consists of a species morphism $\eta : y\Zero \to M$ which by Yoneda is just an element $e\in M\Zero$.
		\item the multiplication splits into a cowedge $\mu_{pq} : M[p]\times M[q] \to M[n]$ for each pair of integers $p,q$ such that $p+q=n$, natural for the action of symmetric groups, under the shuffling maps $S_p\times S_q\to S_{p+q}$ sending a pair of permutations $(\sigma,\tau)$ to the one acting as $\sigma$ on $\{1,\dots,p\}$ and as $\tau$ on $\{p+1,\dots,p+q\}$.
	\end{itemize}
\end{remark}
\begin{remark}
	Let $(M,\mu,\eta)$ be a $\Day$-monoid in $\Spc$; then the slice category $\Spc/M$ is monoidal closed, under a monoidal product $\Day^M$ which makes the forgetful functor $U : \Spc/M\to\Spc$ strong monoidal.

	In fact, there is an indexed monoidal category (cf.~\cite{gouzou1976fibrations,Shulman2008a}) $\Mon(\Spc,\Day)\to\Cat$ sending $M\mapsto\Spc/M$.
\end{remark}
The following is implied joining \cite[Example 2.3]{adamek2008analytic} and adapting \cite[8.16]{aguiar2010monoidal}: in particular, the species $\Lin$ of \autoref{exam_species} has a convenient universal property.
\begin{proposition}\label{bimo_lin}
	(\cite[p. 7]{bergeron1998combinatorial}, \cite[§8.1]{aguiar2010monoidal}) The species $\Lin$ of total orders is the free monoid on $y\Uno$. The species $\Lin_+$ of \emph{nonempty} linear orders is the free semigroup on $y\Uno$. Thus,
	\[\textstyle\Lin\cong\sum_{n\ge 0} y[n]\qquad\qquad \Lin_+\cong\sum_{n\ge 1} y[n].\]
\end{proposition}
(A terminological note. \cite{aguiar2010monoidal} calls `positive' what we tend to dub `nonempty', considering species as monadic over graded vector spaces in a similar fashion of our \Cref{innocent_monadic}.) In fact, in a $k$-linear setting ($k$ a field) the structure of $\Lin$ is way richer: $k\langle\Lin\rangle$ (cf.~\autoref{ch_o_bases}; it's the species assigning to $[n]$ the $k$-vector space having the set $\Lin[n]$ as a basis) carries the structure of a \emph{Hopf monoid}.
Following \autoref{how_a_monoid_in_spc}, the monoid structure of $\Lin$ arises as a cowedge $\Lin[p]\times \Lin[q] \to \Lin[n]$ for every $p+q=n$, defined as $(l,l')\mapsto l\cdot l'$ where the later is the \emph{ordinal sum} or concatenation of the linear orders $l$ on $[p]$ and $l'$ on $[q]$; ordinal sum is an associative operation, equivariant under the shuffling maps of \autoref{how_a_monoid_in_spc}. The unit is the only element of $\Lin\Uno$.

The Hopf monoid structure of $k\langle\Lin\rangle$ is extensively studied and described in \cite[§8.5]{aguiar2010monoidal}.
\subsubsection{Co/algebras for endofunctors of $\Spc$}
This subsection studies algebras and coalgebras for a few interesting endofunctors $M$ defined over $\Spc$. Despite its naturality, this idea is seemingly unexplored thus far, and in particular, no one studied the category of $\de$-algebras (cf.~\autoref{de_alg_def}) outlining the fact that it is a differential 2-rig on its own (cf.~\autoref{anchelui_diffrig}).

It becomes particularly intriguing to explore the interactions between $M$ and the structures on $\Spc$ mentioned in \autoref{various_mono}, \autoref{diff_of_spc}; clearly, this is essential to study $(M,B)$-automata, defined in \autoref{mly_n_mre} as a pullback along $M$-algebras.
\begin{definition}[The category $\Spc^\Lin$]\label{lin_algebras}
	The category $\Spc^\Lin$ is, up to equivalence, described as any of the following:
	\begin{enumtag}{l}
		\item \label{spc_T_1} the category of endofunctor algebras for $y\Uno\Day\firstblank$;
		\item \label{spc_T_2} the category of endofunctor coalgebras for $\de$;
		\item \label{spc_T_3} the \textEM category of the monad $\Lin\Day\firstblank$;
		\item \label{spc_T_4} the co\textEM category of the comonad $\homDay[\Lin]$.
	\end{enumtag}
\end{definition}
These identifications follow from the freeness of $\Lin$ and the general fact that whenever $F\dashv G$ is an adjunction between endofunctors, $\Alg(F)\cong\coAlg(G)$.

Representing objects of $\Spc^\Lin$ as \textEM algebras is particularly convenient, as a $\Lin$-module is the same thing as a $\Day$-monoid homomorphism $\Lin \to \homDay[F][F]$,
which since $\Lin$ is the free monoid generated on $y\Uno$, amounts to a single element of $\homDay[F][F]\Uno$; equivalently, if one uses characterization \ref{spc_T_1} above, a structure of type $y\Uno\Day$ on $[n]$ consists of a choice of point in $[n]$, together with an $F$-structure on the complement of that point.\footnote{One can read off the fact that these descriptions are equivalent from the end defining $\homDay[F][F][n]$, cf.~\cite[Equation (2.6)]{Kelly2005a}.}
\begin{remark}
	Limits and colimits in $\Spc^\Lin$ are computed exactly as in $\Spc$, i.e. pointwise (since $\Spc$ is monadic over $\Set^\bbN = \prod_{n\ge 1}\Set$), given that $\Spc^\Lin$ is at the same time a category of algebras (for $\Lin\Day\firstblank$, hence limits are created in $\Spc$) and of coalgebras (for the right adjoint comonad $\homDay[\Lin]$, hence colimits are created in $\Spc$). We just proved that
\end{remark}
\begin{lemma}\label{term_in_spc}
	The terminal object of $\Spc^\Lin$ is the exponential species of \autoref{exam_species}, whence the isomorphism $\de \Exp\cong \Exp$ characterizing $\Exp$ as a `Napier object' of the differential 2-rig of species.\footnote{The rationale behind the terminology is that, evidently, `exponential object' already has a different, conflicting meaning.}
\end{lemma}
Armed with these explicit computations, we can attempt to unveil the structure of the category $\Spc^\Lin$ in any of the equivalent forms given in \autoref{lin_algebras} as a building block of $\Mly[\Spc](\Lin,\firstblank)$.

We now collect some examples of: a species that has only a few structures of $\Lin$-algebra (=structures of $\de$-coalgebra); a species that has at least uncountably many; a species with \emph{no} such structure as a $\Set$-species, that however becomes interesting when `changing base' (cf.~\autoref{ch_o_bases}).
\begin{example}\label{ex_alg_1}
	Structures of $\de$-coalgebra on the species of subsets of \ref{es_1} correspond to $S_n$-equivariant maps $\theta : \wp\to\de \wp$ and using the Leibniz rule over the isomorphism $\wp\cong E\Day E$ of \cite[§1.3, Eq. (33)]{bergeron1998combinatorial} one gets that $\theta : \wp \to \wp+\wp$. Using elementary group theory on the components $\theta_A$ one sees that there are only four such $\theta$: embedding a subset $U\subseteq A$ in the first summand, embedding a subset $U\subseteq A$ in the second summand, embedding $U^c = A\smallsetminus U$ in the first summand, embedding $U^c = A\smallsetminus U$ in the second summand.
\end{example}
\begin{example}\label{ex_alg_2}
	\cite[Example 9, (37)]{bergeron1998combinatorial} yields $\de\Lin\cong \Lin\Day\Lin$, whence a natural choice for a coalgebra structure $s : \Lin\to\de \Lin$, given a finite set $A$, is specified on components $s_A$ in terms of a choice of decomposition $A = I\sqcup J$ and a splitting of the total order on $A$ as a total order on $I$ and a total order on $J$. This choice is made independently for every finite set $A$, so this argument shows that there is an uncountable infinity of coalgebra structures on $\Lin$.
\end{example}
\begin{example}\label{ex_alg_3}
	Let $\Cyc$ be the species of cyclic orders, \autoref{exam_species}.\ref{es_4}; then, we immediately get $\de \Cyc\cong\Lin$ from manipulating generating series. A $\de$-coalgebra structure on $\Cyc$ now would be a natural transformation $\vartheta : \Cyc\to \Lin$, and no such map can exist by cardinality reasons: since $\Cyc[n]$ identifies with the coset space $S_n/\bbZ_n$, over which $S_n$ acts transitively, an $S_n$-equivariant map $\vartheta_n : \Cyc[n] \to S_n$ must be surjective (the translation action $S_n\times \Cyc[n]\to \Cyc[n] : (\sigma,\tau)\mapsto \sigma\tau$ is also transitive). Yet, $|S_n| = n! > (n-1)! = |\Cyc[n]|$.
\end{example}
\begin{example}\label{ex_alg_4}
	Let $\fkS$ be the species of permutations of \autoref{exam_species}.\ref{es_3}; from \autoref{important_comb_ids} it follows that $\de \fkS\cong \fkS\Day \Lin$, so that $\de$-coalgebra structures (i.e. \textEM algebras for $\Lin\Day\firstblank$) correspond under adjunction to monoid homomorphisms $\Lin \to \homDay[\fkS][\fkS]$.
\end{example}

\section{Categories of automata}\label{sec_abs_auto_in_spc}
In this section we define categories of automata, both in the generalised Ad\'amek-Trnkov\'a sense, \autoref{mly_n_mre}, and in the monoidal sense, \Cref{mon_tot_mach}, with particular care in outlining the fibrational properties of the correspondence $(F,B)\mapsto\Mly(F,B)$.

The typical object of such category depends parametrically on a pair $F,B$ where $F : \clK\to\clK$ is an endofunctor of a (in many instances, symmetric monoidal) category $\clK$ and $B\in\clK$ is an object; a \emph{Mealy automaton} $\mlyob Xds$ is then a span of the following form
\[\label{obj_mly}\vxy{
		X & FX \ar[r]^-s \ar[l]_-d & B.
	}\]
These spans are the objects of a category with morphisms $f : \mlyob Xds\to \mlyob Y{d'}{s'}$ those $f : X\to Y$ that are, at the same time, morphisms of $F$-algebras and `fibered' over $B$, meaning that
\[\label{mor_mly}f\circ d  = d'\circ Ff\qquad \text{and} \qquad s = s'\circ Ff. \]
A Moore automaton is defined similarly, just instead of being a span it's a disconnected diagram
\[\label{obj_mre}\vxy{
		X & FX \ar[l]_-d & X \ar[r]^-s  & B.
	}\]
The endofunctor $F : \clK \to \clK$ has to be understood as an abstraction of a dynamical system through iteration $F,F^2,F^3,\dots : \clK\to\clK$ --this is the point of view of~\cite{adam-trnk:automata}. 

We also fix an object $B\in\clK$ (an `output' object, cf.~\cite{Ehrig,Guitart1980}).

The explicit descriptions given in \Cref{obj_mly}, \eqref{mor_mly}, \eqref{obj_mre} makes it evident that the categories in study have the following universal properties.
\begin{definition}\label{mly_n_mre}
	We define the category $\Mly(F,B)$ of \emph{Mealy automata with input $F$ and output $B$} and $\Mre(F,B)$ of \emph{Moore automata with input $F$ and output $B$} as the following strict 2-pullbacks in $\Cat$ respectively:
	\[\label{eq_mly_n_mre}\vxy{
			\Mly(F,B) \ar[r]\ar[d]
			\ar@{}[dr]|(.375){\dopb}
			& F/B \ar[d]^U& \Mre(F,B) \ar[r]\ar[d]
			\ar@{}[dr]|(.375){\dopb}
			& \clK/B \ar[d]^{U'}\\
			\Alg(F)\ar[r]_V & \clK & \Alg(F) \ar[r]_V & \clK
		}\]
	where $\Alg(F)$ is the category of \emph{endofunctor algebras} of $F$, $F/B$ the comma category of arrows $FX\to B$, and $\clK/B$ the comma category of arrows over $B$, \ie $u : X\to B$ (and $U,V,U',V'$ are the most obvious forgetful functors).
\end{definition}
\begin{remark}[Limits and colimits in categories of automata]\label{remark_on_compl_1}
	If $F$ admits a right adjoint $R$, and $\clK$ is complete and cocomplete, so are $\Mly(F,B)$ and $\Mre(F,B)$; this can be easily argued using an argument in \cite[V.6, Ex. 3]{working-categories} and the fact that $U,U'$ create colimits and connected limits, together with the fact that $F/B\cong\clK/RB$; then, one easily verify by inspection that the terminal object of $\Mly(F,B)$ is $\prod_{n\ge 1} R^nB$ and the terminal object of $\Mre(F,B)$ is $\prod_{n\ge 0}R^n B$ (note how they only differ by a shift of index).
\end{remark}
\begin{remark}[Accessibility of categories of automata]\label{remark_on_compl_2}
	Repeatedly applying the completeness theorem of the 2-category $\mathbf{Acc}$ of accessible categories \cite[Ch. 5]{makkai1989accessible} one can prove that if $\clK$ is locally presentable (say for a regular cardinal $\kappa$) and $F$ is $\kappa$-accessible (clearly an assumption subsumed by its being a left adjoint), then $\Mly(F,B),\Mre(F,B)$ are both locally presentable (but in general, for a much higher cardinal $\kappa$).
\end{remark}
\begin{remark}\label{remark_on_compl_3}
	A particular instance of \autoref{remark_on_compl_1} is when $\clK$ is monoidal and $F:\clK\to\clK$ is the tensor product $A\otimes-$ for a fixed object of $\clK$. Then, we shorten $\Mly(F,B)$ and $\Mre(F,B)$ to $\Mly(A,B)$ and $\Mre(A,B)$ and we observe that
	\begin{itemize}
		\item if $\clK$ has countable sums, $\Alg(F)=\Alg(A\otimes-)$ is the Eilenberg-Moore category of the monad $A^*\otimes-$ where $A^*:=\sum_{n=0}^\infty A^{\otimes n}$ is the free monoid on $A$;
		\item all the results stated so far specialize: if $\clK$ is monoidal closed, complete and cocomplete, then $\Mly(A,B)$ and $\Mre(A,B)$ are complete and cocomplete; if $\clK$ is locally $\kappa$-presentable, so are $\Mly(A,B)$ and $\Mre(A,B)$ (generally, for a larger cardinal $\kappa'\gg\kappa$). The terminal object in $\Mly(A,B)$ is $[A^+,B]$, $A^+$ being the free semigroup on $A$ (resp., in $\Mre(A,B)$ it's $[A^*,B]$, $A^*$ being the free monoid).
	\end{itemize}
	Unwinding \autoref{mly_n_mre} in this particular case, the typical object $\mlyob Eds$ of $\Mly(A,B)$ is a span as in the left of the following diagram, and the typical object $\mreob Eds$ of $\Mre(A,B)$ a (disconnected) diagram as in the right
	\[\label{useful_notaiza}\vxy{
			\mlyob Eds \,: E & \ar[l]_-d A\otimes E \ar[r]^-s & B & \mreob Eds\, : E & \ar[l]_-d A\otimes E, E \ar[r]^-s & B.
		}\]
	Such models of computation with the final state depending (on the left) or not depending (on the right) from the inputs $A$ has a long history, cf.~\cite{6771467,Moore1956,5222698}. Its categorical axiomatization also has a long tradition, cf.~\cite{Goguen1972a,Goguen1973,Arbib1975}.
\end{remark}
The general observations collected so far specialize to the category of \autoref{spc_n_Vspc} as follows.
\begin{remark}
	Remarks \ref{remark_on_compl_1}, \ref{remark_on_compl_2}, \ref{remark_on_compl_3} all apply to $\clK=\Spc$ considered with the Day convolution structure (and in fact to all $\clV\emdash\Spc$ when $\clV$ is complete, cocomplete and monoidal closed). In particular, for every fixed combinatorial species $B : \sfP \to\Set$ we can easily study $\Mly[\Spc](L,B)=\Mly[\Spc](y\Uno,B)$ as the category having objects the diagrams $E \xot d y\Uno\Day E \xto s B$, or more concisely as the category obtained as the pullback $\Spc^\Lin\times_\Spc (\Spc/B)$ where $\Spc^\Lin$ is as in \autoref{lin_algebras}.
\end{remark}
Note that this is equivalent to the category of coalgebras for the functor $E\mapsto \de B\times \de E$. From this coalgebraic characterization, we deduce that
\begin{proposition}\label{omega_limit}
	The terminal object of $\Mly[\Spc](L,B)$ is the `$\omega$-differential limit'\footnote{The name is borrowed from ergodic theory and it is chosen in analogy with the notion of $\omega$-limit set of a dynamical system $f : X\to X$ defined over a metric space, see e.g. \cite[Def. 1.12]{book_91686061}, where the ($\omega$-)limit set of $x$ under $f$ is defined as $$\textstyle\omega(x,f) = \bigcap_{n\in \bbN} \overline{\{f^k(x): k>n\}},$$ the topological closure of the `eventual $f$-orbits' of $x$.} of $B$ defined as
	\[\textstyle\prod_{n\ge 1}\de^n B \cong \prod_{n\ge 1}\{y\Uno^{\Day n},B\}_\mathrm{Day}\cong\left\{\sum_{n\ge 1}y[n],B\right\}_\mathrm{Day}=\{y\Uno^+,B\}_\mathrm{Day}\]
	where again $y\Uno^+$ is the free semigroup on $y\Uno$: given \autoref{bimo_lin}, $y\Uno^+\cong\Lin_+$.
\end{proposition}
\subsection{Fibrational properties of the \texorpdfstring{$\Mly$}{Mly} construction}\label{fib_propes}
We can construct total categories where \emph{all} dynamics and outputs can be considered simultaneously and coherently. This is a consequence of the well-known correspondence between indexed categories $\clC\to\mathbf{Cat}$ out of a domain $\clC$ and fibrations $\clE\to \clC$ over $\clC$.

\begin{remark}\label{mly_functor_inK}
	Consider two endofunctors $F : \clK\to\clK$, $G:\clH\to\clH$. If $P :\clK\to\clH$ is a functor intertwining $F,G$, \ie equipped with a natural transformation $\pi : GP\To PF$ we can define a functor $\pi^*:\Alg(F)\to\Alg(G)$ by application of $P$ and precomposition with $\pi$, a functor $\clK/B\to\clH/PB$ in the obvious way, and in turn a unique functor
	\[\vxy{\varpi^*:\Mly(F,B) \ar[r] & \Mly[\clH](G,PB).}\]
\end{remark}
This simple observation paves the way to the fruitful consideration that the construction in \autoref{mly_n_mre} is implicitly functorial in the pair $F,B$, and, in turn, motivate the interest in the properties of the pseudofunctor $(F,B)\mapsto \Mly(F,B)$.
\begin{definition}\label{glob_Mly_cat}
	The \emph{total Mealy 2-category} $\mathbf{Mly}$ is defined as follows:
	\begin{itemize}
		\item the objects are triples $(\clK;F,B)$ where $F : \clK\to\clK$ is an endofunctor of a category $\clK$, and $B$ an object of $\clK$;
		\item the morphisms $(P,\pi,u):(\clK;F,B)\to(\clH;G,B')$ are triples where $P : \clK\to\clH$ is a functor, $\pi : GP\To PF$ is an \emph{intertwiner} natural transformation between $F$ and $G$ and $u : PB\to B'$ is a morphism;
		\item 2-cells $\gamma : (P,\pi,u)\To(Q,\theta,v)$ consist of natural transformations $\gamma : P\To Q$ compatible with the intertwiners $\pi,\theta$ in the obvious sense, and such that $v\circ\gamma_B=u$.
	\end{itemize}
	From such a domain $\mathbf{Mly}$, sending $(\clK,F,B)$ to $\Mly(F,B)$ results in a strict 2-functor $\mathbf{Mly} \to\mathbf{Cat}$ ($\mathbf{Cat}$ is the 2-category of categories, strict functors, strict natural transformations).
\end{definition}
It is, however, rarely needed to vary the domain $\clK$ of the automata in study (but cf.~\autoref{unrewarding} for an instance of when this `change of scalars' might be required). A simpler (=lower-dimensional) approach is convenient if we are content with keeping $\clK$ fixed.
\begin{definition}[The total categories of automata]\label{tot_mach}
	\autoref{mly_n_mre} entails at once that the correspondence $(F,B)\mapsto\Mly(F,B)$ is a (pseudo)functor of type $\Mly : \Cat(\clK,\clK)^\op\times \clK \to \Cat$,
	i.e. a pseudo-profunctor $\Cat(\clK,\clK)\pto\clK$ from which we can extract a two-sided fibration, \ie a span
	\[\vxy{\Cat(\clK,\clK) &\ar[l]_-p \clMly \ar[r]^-q & \clK}\]
	such that $p$ is a fibration, $q$ is an opfibration, $p$-Cartesian lifts are $q$-vertical and $q$-opCartesian lifts are $p$-vertical. The tip $\clMly$ of the span, we call the \emph{total Mealy category} constructed from $\clK$.

	Similar considerations allow to construct the total Moore category $\clMre$ from the pseudo-profunctor $(F,B)\mapsto\Mre(F,B)$, and obtain a two-sided fibration $\Cat(\clK,\clK)\leftarrow\clMre\to\clK$, the \emph{total Moore category}.
\end{definition}
\begin{remark}
	Unwinding the definition, it is easy to establish how reindexings of the total Mealy and Moore fibration act. In the particular case where $\alpha : F\To G$ is a natural transformation between left adjoints $F\dashv R$ and $G\dashv Q$ and $f : B\to B'$ a morphism, the reindexing functor $\clMly(\alpha,f) : \clMly(G,B)\to\clMly(F,B')$ preserves all colimits --and thus, assuming $\clK$ is a locally presentable category, is a left adjoint; however, it fails to preserve limits (it already fails to preserve terminal objects; such behaviour can be put in perspective, once the coalgebraic nature of $\clMly$ is unraveled: cf.~\autoref{its_always_a_foa}).\footnote{It is probably interesting to devise under which conditions the canonical map $(\alpha,f)^*\big(\prod_{n\ge 1} Q^nB\big)\to \prod_{n\ge 1}R^nB'$, is well behaved in some sense (for example, under the mild condition that there exist at least one `point' in its domain, the map is a split epi).}
\end{remark}
If $\clK$ is monoidal its tensor functor $\firstblank\otimes-:\clK\times\clK \to \clK$ now curries to the `left regular representation' $\lambda:\clK \to\Cat(\clK,\clK):A\mapsto A\otimes-$ of $\clK$ on itself, and as a consequence, we can pullback the total Mealy fibration and the total Moore fibration to obtain the left leg of the diagram
\[\label{mon_tot_mach}\vxy{
	\clMlyTens\ar[r]\ar[d]
	\ar@{}[dr]|(.375){\dopb}
	& \clMly \ar[d]\\
	\clK^\op\times\clK\ar[r]_-{\lambda^\op\times\clK} & \Cat(\clK,\clK)^\op\times\clK
	}\]
which gives rise to the \emph{monoidal Mealy} (two-sided) \emph{fibration}
\[\label{monomealy_fib}\vxy{
	\clK & \clMlyTens \ar[r]^-{q^\otimes} \ar[l]_-{p^\otimes} & \clK
	}\]
(Similar considerations define $\clMreTens$, but we refrain from doing so for some technical reasons that make $\clMlyTens$ a better-behaved object than $\clMreTens$, cf.~\cite{boccali2023semibicategory}.) In fact, the terminology is chosen to inspire the fact that we have restricted the total Mealy category to the case where $F$-actions are monoidal and hint at the following result.
\begin{proposition}\label{its_monoidal}
	The monoidal Mealy fibration is a \emph{monoidal two-sided fibration}, in the sense of \cite{Yoneda,StreetFibreYoneda1974}, and the monoidal product interfiber is given by componentwise tensor product,
	\[\big(A,B;\mlyob Eds\big)\otimes \big(A',B', \mlyob {E'}{d'}{s'}\big) = \big(A\otimes A', B\otimes B';\mlyob{E\otimes E'}{d\otimes d'}{s\otimes s'}\big)\]
\end{proposition}
\begin{theorem}[\protect{\cite{Katis2010}, \cite[Def. 1]{ROSEBRUGH_SABADINI_WALTERS_1998}} rephrased]\label{its_promonad}
	If $\clK$ is \emph{Cartesian} monoidal, the profunctor $\clK^\op\times\clK \to\Cat$ obtained from \Cref{monomealy_fib} carries the structure of a (pseudo)promonad, and it gives rise to a bicategory $\Mly$ whose hom-categories are precisely the $\Mly(A,B)$.
\end{theorem}
\section{The differential structure of $\texorpdfstring{\Mly[\Spc]}{Mly}$}
The scope of this section is to study the differential 2-rig $(\Spc,\Day,\de)$ in more depth, by extending the general theory of differential 2-rigs.	The terminology introduced so far gives us enough leeway to introduce the main theorem of the present section: categories of automata on a differential 2-rig $(\clK,\otimes,\de)$ form themselves a differential 2-rig, such that the functor of \Cref{monomealy_fib} is a \emph{fibration of differential 2-rigs} (=a strong monoidal functor, preserving the differential, which moreover is a fibration).
\begin{theorem}\label{its_diffe}
	Let $(\clK,\otimes,\de)$ be a differential 2-rig; then the total category of the monoidal Mealy fibration is itself a differential 2-rig for a canonical choice of a derivative functor $\bar\de : \clMlyTens \to \clMlyTens$ such that the projection functors $p^\otimes,q^\otimes$ in \Cref{monomealy_fib} are (strict) morphisms of differential 2-rigs.

	A similar statement holds replacing $\clMlyTens$ with the category $\clMreTens$.
\end{theorem}
We conduct the proof in full detail, in the case of $\clMlyTens$; the proof for $\clMreTens$ is analogous, \emph{mutatis mutandis}.
\begin{lemma}\label{its_always_a_foa}
	The monoidal Mealy fibration $\clMlyTens$ of \Cref{mon_tot_mach} arises as the category of coalgebras for a endofunctor $R$ of $\clK^\op\times\clK\times\clK$, sliced over the projection $\pi_{12} : (\clK^\op\times\clK)\times\clK \to \clK^\op\times\clK$ on the first two factors; this means $R : (\clK^\op\times\clK\times\clK,\pi_{12}) \to (\clK^\op\times\clK\times\clK,\pi_{12})$ is a morphism in the slice $\Cat/(\clK^\op\times\clK)$, making the triangle
	\[\vxy{
		(\clK^\op\times\clK)\times\clK\ar[dr]_{\pi_{12}}\ar[rr]^-R && (\clK^\op\times\clK)\times\clK \ar[dl]^{\pi_{12}}\\
		& \clK^\op\times\clK
		}\]
	commute strictly.
\end{lemma}
\begin{proof}
	The functor $R$ is defined as
	\[\vxy[@R=0cm]{
		(\clK^\op\times\clK)\times\clK \ar[rr]&& (\clK^\op\times\clK)\times\clK\\
		(A,B;X) \ar@{|->}[rr] && (A,B;R_{AB}X)
		}\]
	where $R_{AB}X = [A, X\times B]$; it is well-known that a Mealy automaton $\mlyob Xds$ in $\Mly(A,B)$ is a $R_{AB}$-coalgebra, cf.~\cite[Exercise 2.3.2]{Jacobs2016}, and then the result (\ie the fact that $\clMlyTens$ is the object of $R$-coalgebras in $\Cat/(\clK^\op\times\clK)$) follows from (the dual of) \cite[Remark 3.3]{fib_of_alg}, if the category $\clK^\op\times\clK$ is treated as a category of parameters.
\end{proof}
An immediate consequence of \autoref{its_always_a_foa} is that colimits in $\clMlyTens$ can be computed as in the base $\clK$ (coalgebra objects are inserters of the form $\Ins(\id[],R)$, whose forgetful functor $V : \Ins(\id[],R) \to \clK$ create colimits):
\begin{corollary}
	Colimits in $\clMlyTens$ are created by a canonical forgetful functor
	\[\vxy{V : \clMlyTens \ar[r] & (\clK^\op\times\clK)\times\clK}\]
	presenting its domain $\clMlyTens$ as the inserter $\Ins(\id[],R)$ \cite[(4.1)]{2catlimits}, \ie as a square
	\[\vxy{
			\clMlyTens \drtwocell<\omit>{\upsilon} \ar[r]^-V\ar[d]_V& (\clK^\op\times\clK)\times\clK \ar@{=}[d]\\
			(\clK^\op\times\clK)\times\clK \ar[r]_-R & (\clK^\op\times\clK)\times\clK
		}\]
	terminal among all such.
\end{corollary}
\begin{remark}\label{how_colimits_work}
	It is worth to unravel how colimits are indeed computed in $\clMlyTens$, and in particular make the construction of coproducts explicit, as they will be needed to prove linearity and the Leibniz property of $\bar\de$. Recall the compact notation in \Cref{useful_notaiza}: an object of $\clMlyTens$ is denoted as $\mlyob Eds$.

	Given a diagram $H : \clJ \to \clMlyTens : J\mapsto \mlyob {X_J}{d_J}{s_J}_{A_JB_J}$, let $A:=\lim_J A_J$ (with terminal cone $(\pi_J : A\to A_J\mid J\in\clJ)$) and $B := \colim_J B_J$ (with initial cocone $(\iota_J : B_J\to B\mid J\in\clJ)$): then, the diagram
	\[\label{pushato_forwardo}\clJ \to \clMlyTens : J\mapsto (\pi_J,\iota_J)_*\mlyob{X_J}{d_J}{s_J}\]
	obtained from the reindexing functor has shape $\clJ$ and it lives entirely in the fiber over $(A,B)$. The colimit of \Cref{pushato_forwardo} can then be computed in this fiber, and it is a matter of elementary diagram-chasing to show that this is a colimit for the diagram $H$ in the whole $\clMlyTens$.

	In the specific case of (binary, and by induction, finite) coproducts, this construction starts with two objects $\mlyob Xds_{AB}$ and $\mlyob Y{d'}{s'}_{A',B'}$, builds the diagram
	\[\vxy[@R=0cm]{
		A & A\times A'\ar[r]^-{\pi'}\ar[l]_-\pi & A' \\
		B \ar[r]^-\iota & B+B' & B'\ar[l]_-{\iota'}
		}\]
	which is a colimit diagram in $\clK^\op\times\clK$, and then pushes $\mlyob Xds,\mlyob Y{d'}{s'}$ forward into the fibre $\Mly(A\times A', B+B')$ using the reindexings
	\[\vxy[@C=2cm]{
		\Mly(A,B) \ar[r]_-{(\pi,\iota)_*} & \Mly(A\times A', B+B') & \ar[l]^-{(\pi',\iota')_*}\Mly(A',B').
		}\]
	Then, one computes the coproduct in the fibre $\Mly(A\times A', B+B')$, \ie the diagram
	\[\vxy{
		&(A\times A')\otimes (X+Y)\ar@{=}[d]\\
		&(A\times A')\otimes X + (A\times A')\otimes Y \ar[dr]|{\iota\circ s \circ (\pi\otimes X) +\iota'\circ s'\circ (\pi'\otimes X) }\ar[dl]|{d\circ (\pi\otimes X)+d'\circ (\pi'\otimes X)}\\
		X+Y && B+B'.
		}\]
	It is a lengthy but easy computation to see that this construction satisfies the universal property of coproducts in $\clMlyTens$.
\end{remark}
If $\clX$ is any category and $R$ an endofunctor of $\clX$, it is well-known that liftings of an endofunctor $D : \clX\to\clX$ to the category of $R$-coalgebras correspond bijectively to distributive laws $\lambda : D R\To R D$; we now want to deduce the existence of the former lifting from the existence of the latter distributive law, when the product category $\clK^\op\times\clK$ becomes a differential 2-rig under the action
\[\label{tolift_1}\vxy{D=\id\times\de\times\de : \clK^\op\times\clK\times\clK \ar[r] & \clK^\op\times\clK\times\clK}\]
sending $(A,B,X)\mapsto (A,\de B,\de X)$ (this corresponds to deriving the carrier and output objects, but not the input $A$).
\begin{construction}
	In \Cref{tolift_1} the distributive law is the natural transformation with components
	\[\vxy[@R=0cm@C=2cm]{(\id\times\de\times\de)\circ R \ar[r]& R\circ(\id\times\de\times\de)\\
		(A,\de B,\de R_{AB}X) \ar@{|->}[r]_{(\id[A],\id[\de B],\lambda)} & (A,\de B,R_{A,\de B}\de X)
		}\]
	where $\lambda : \de [A,X\times B]\to [A,\de X\times \de B]$ is obtained from the tensorial strength of $\de$
	\[\label{stren}\vxy{\de [A,X\times B] \ar[r]^-\star & [A,\de(X\times B)] \ar[rrr]^{[A,\langle\de\pi_X,\de\pi_B\rangle]} &&& [A,\de X\times \de B]}\]
	and the arrow $\star$ is obtained as composition obtained from the tensorial strength of $\de$ and the monoidal closed adjunction $-\otimes A\dashv [A,-]$,
	\[\vxy{
		\de[A,X\times B] \ar[r]^-{\eta} & [A,\de[A,X\times B]\otimes A] \ar[d]^{[A,\tau^\textsc{l}]}\\
		& [A,\de([A,X\times B]\otimes A)] \ar[r]_-{\de\epsilon} & [A,\de(X\times B)]
		}\]
	as in \cite{kock1972strong}.
\end{construction}
The arrow in \Cref{stren} now yields the lifting of $\de$ to a functor
\[\label{xavier_barde}\vxy{\bar\de : \clMlyTens \ar[r] & \clMlyTens}\]
defined sending $\mlyob Xds_{AB}$ to the span $\mlyob {\de X}{\de d\circ i,\de s\circ i}{}_{A,\de B}$,
\[\vxy{
	& A\otimes \de X \ar[dr]\ar[dl]\ar[d]_i\\
	\de x & \de(A\otimes X) \ar[r]^-{\de s} \ar[l]_-{\de d}& \de B
	}\]
(where $i$ is the right tensorial strength of $\de$) which is
\begin{itemize}
	\item linear, because given two objects $\mlyob Xds$ and $\mlyob Y{d'}{s'}$, denoting
	      \[\vxy{d\oplus d' : (A\times A')\otimes(X+Y) \ar[r]^-\cong & (A\times A')\otimes X + (A\times A')\otimes Y \ar[r]^-{d(\pi\otimes X)+d'(\pi'\otimes Y)} & X + Y}\]
	      and similarly for $s\oplus s'$, one has an isomorphism
	      \begin{align*}
		      \bar\de\left[(\pi,\iota)_*\mlyob{X}{d}{s} +_{\substack{A\times A' \\B+B'}} (\pi',\iota')_*\mlyob{Y}{d'}{s'} \right] &\cong \bar\de \mlyob{X+Y}{d\oplus d'}{s\oplus s'}\cong \mlyob{\de(X+Y)}{\de(d\oplus d')\circ i}{\de(s\oplus s')\circ i}\\ &\textstyle\cong\mlyob{\de X}{\de d\circ i}{\de s\circ i}_{A,\de B} + \mlyob{\de Y}{\de d'\circ i}{\de s'\circ i}_{A',\de B'}\\ & \cong (\pi,\iota)_*\mlyob{\de X}{\de d\circ i}{\de s\circ i}_{A,\de B} +_{\substack{A\times A'\\\de B+\de B'}} (\pi',\iota')_*\mlyob{\de Y}{\de d'\circ i}{\de s'\circ i}_{A',\de B'}
	      \end{align*}
	      (the proof is an exercise in casting the universal property, made painstaking by the definition of morphism in the fibered category $\clMlyTens$ implicitly given in \Cref{monomealy_fib}).
	\item Leibniz, because once the monoidal structure on $\clMlyTens$ is defined as in \Cref{its_monoidal},
	      one has a tensorial strength on $\bar\de$ given by
	      \[\vxy{
			      \bar\tau_{(X,d,s),(Y,d',s')}^\textsc{r} := (\id[A\otimes A'],\tau_{B,B'};\tau_{XY}) :
			      \mlyob{X\otimes\de Y}{d\otimes (\de d'\circ i)}{s\otimes(\de s'\circ i)}\ar[r] & \mlyob{\de(X\otimes Y)}{\de(d\otimes d')}{\de(s\otimes s')}
		      }\]
	      The proof that these components are indeed morphisms in $\clMlyTens$ relies on the naturality of $\tau$'s components, as well as their compatibility with themselves: for example, every part of diagram
	      \[\small\notag\vxy{
		      {\partial X\otimes Y} \ar[ddd]_{\tau_{XY}}& {\partial(A\otimes X)\otimes A'\otimes Y} \ar[ddd]_{\tau_{A\otimes X,A'\otimes Y}} \ar[l]_-{\partial d \otimes d'}& {A\otimes \partial X\otimes A'\otimes Y} \ar@{=}[r]\ar[l]_-{\tau_{AX}\otimes A'\otimes Y}\ar[d] & {A\otimes A'\otimes \partial X\otimes Y }\ar[ddd]^{A\otimes A'\otimes \tau_{XY}} \\
		      && {\partial(A\otimes X)\otimes A'\otimes Y }\ar[d] \\
		      && {\partial(A\otimes X \otimes A'\otimes Y)} \ar@{=}[d] \\
		      {\partial(X\otimes Y)} &\ar[l]_-{\de(d\otimes d')}\ar@{=}[r] {\partial(A\otimes X\otimes A'\otimes Y)} & {\partial(A\otimes A'\otimes X\otimes Y)} & {A\otimes A'\otimes \partial(X\otimes Y)}\ar[l]_-{\tau_{A\otimes A',X\otimes Y}}
		      }\]
	      commutes by the axioms of tensorial strength.

	      Using the left strength $\tau^\textsc{l}$ of $\de$ one defines the left strength $\bar\tau^\textsc{l}$ of $\bar\de$ in a similar fashion, and thus a unique map which is a candidate leibnizator (obtained from the leibnizator of $\de$, of course). The construction of coproducts in $\clMlyTens$, and the specific way in which the tensorial strength for $\bar\de$ is induced pushing forward with the tensorial strength components of $\de$ now entails that the diagram
	      \[\vxy{
			      \mlyob Xds\otimes\bar\de\mlyob Y{d'}{s'} \ar[r]& \mlyob{\de(X\otimes Y)}{\de(d\otimes d')}{\de(s\otimes s')} & \bar\de\mlyob Xds\otimes \mlyob Y{d'}{s'}\ar[l]
		      }\]
	      is a coproduct in $\clMlyTens$, thus proving the invertibility of $\bar\fkl$.
\end{itemize}
In a completely analogous fashion, one proves similar results for the general opfibration associated to $[\clK,\clK]^\op\times\clK \to \Cat$ sending $(F,B)\mapsto \Mly(F,B)$ (the result is probably too general to be of some use when $F$ is free to vary over all $[\clK,\clK]$, so it can be restated in terms of the opfibration $B\mapsto \Mly(F,B)$ alone, obtained fixing the first argument of $\Mly(-,-)$):
\begin{proposition}
	Let $\LAdj[\clK,\clK]$ be the full subcategory of left adjoint endofunctors of $\clK$. Then the category $\clMly$ of \autoref{tot_mach}, appropriately restricted over $\LAdj[\clK,\clK]$, can be seen as the object of coalgebras for a parametric functor $\Pi : \LAdj[\clK,\clK]^\op\times\clK \times\clK \to\clK$, precisely the parameteric functor sending $(F,B,X)\mapsto (F,B,R_F(X\times B))$ where $F\dashv R_F$.
\end{proposition}
\begin{theorem}
	Let $(\clK,\otimes,\de)$ be a differential 2-rig; denote $\LAdj[\clK,\clK]_{\otimes,l}$ the subcategory of \emph{lax monoidal} left adjoint functors $\clK\to\clK$; again suitably restricting $\clMly$ to be fibered over $\LAdj[\clK,\clK]_{\otimes,l}$, for every distributive law $\lambda : F\de\To\de F$ of $F$ over $\de$ we find a lifting of the derivative $\de$ to a derivative $\bar\de : \clMly\to\clMly$, defined on components as
	\[\vxy{
			\bar\de : \Mly(F,B) \ar[r] & \Mly(F,\de B)
		}\]
	by sending $\mlyob Xds$ to the `precomposition with $\lambda$':
	\[\vxy{
		\de X & \ar[l]_-{\de d}\de F X & F\de X\ar[r]^-{\lambda_X} \ar[l]_-{\lambda_X} & \de F X \ar[r]^-{\de s} & \de B.
		}\]
\end{theorem}
Let's observe that linearity of $\bar\de$ can be extended to preservation of all colimits preserved by $\de$ in the base: the proof goes as for coproducts, and uses the explicit description of colimits given in \autoref{how_colimits_work}. Thus we obtain at once
\begin{corollary}
	The total category $\clMly[\Spc]^\otimes$ constructed in \autoref{tot_mach}, \Cref{mon_tot_mach} (underlying category), and \autoref{its_monoidal} (monoidal structure) is a differential 2-rig with respect to the functor $\bar\de$ defined as in \autoref{xavier_barde}, and $\bar\de$ commutes with all colimits.
\end{corollary}
\begin{proposition}\label{its_prese}
	The category $\clMly[\Spc]^\otimes$ is locally presentable, so by the special adjoint functor theorem (cf.~\cite[§3.3]{Bor1}) $\bar\de$ has a left adjoint; in fact, more is true:
	\begin{itemize}
		\item the \emph{fibration} of \Cref{monomealy_fib} is accessible (and cocomplete, hence locally presentable) in the sense of \cite[5.3.1]{makkai1989accessible}, i.e. the total category $\clMly[\Spc]^\otimes$ is locally presentable, the projection $\langle p,q\rangle$, all reindexing functors are accessible, and the pseudofunctor associated to the fibration preserves filtered colimits.
		\item the $\bar\de$ functor is also continuous, hence $(\clMly[\Spc]^\otimes,\Day,\bar\de)$ is a \emph{scopic} differential 2-rig in the sense of \autoref{scopic_2_rig}.
	\end{itemize}
\end{proposition}
\begin{proof}
	The only verification that is not completely immediate is that $\bar\de$ preserves all limits; this can be reduced to the verification that $\bar\de$ preserves the terminal object as described in \autoref{remark_on_compl_1}, because connected limits are created by the forgetful functor to $\Spc$. We then have to establish an isomorphism
	\[\textstyle\bar\de\mlyob{\homDay[A^*][B]}{d_\infty}{s_\infty}_{AB} \cong \mlyob{\de\homDay[A^*][B]}{\de d_\infty \circ i}{\de s_\infty\circ i}\cong \mlyob{\homDay[A^*][\de B]}{d_\infty'}{s_\infty'}_{A,\de B}.\]
	Note that for every $X,Y\in\Spc$ there is an isomorphism $\theta : \de\homDay[X][Y]\cong \homDay[X][\de Y]$, induced by the fact that these two objects are isomorphic if and only if for every $X,Y$ one has $L(X\Day Y)\cong LX\Day Y$ (which is obviously true considering that $L\cong y\Uno\Day\firstblank$, cf.~\autoref{flesh_deriv} and \ref{mangusta}). One can then verify that $\theta$ is indeed an isomorphism in $\clMly[\Spc]^\otimes$.
\end{proof}
The following lemma splits the verification that $\Mly$, defined in \autoref{tot_mach}, preserves filtered colimits in both components into two parts. As an immediate corollary (filtered categories are sifted), $\Mly$ preserves filtered colimits.

Verifying the first part is a straightforward consequence of the fact that $\Alg(\firstblank)$ preserves filtered colimits, in the sense that if $\clJ$ is a $\lambda$-filtered category, $\Alg(\colim_\clJ F_i)\cong \lim_\clJ\Alg(F_i)$. The key result allowing us to prove the second part is the fact that $R$ as described in \autoref{descr_di_LR} also preserves filtered colimits, hence for every filtered diagram, one has an isomorphism of comma categories $(F/\colim_\clJ B_i)\cong \colim_\clJ \big(F/B_i\big)$.
\begin{lemma}
	For every fixed output object $B\in\clK$, the functor $\Mly(-,B)$ preserves filtered colimits. For every fixed dynamics $F : \clK\to\clK$, the functor $\Mly(F,-)$ preserves filtered colimits.
\end{lemma}
\subsection{The structure of $\texorpdfstring{\Alg(\de)}{AlgD}$}\label{struc_of_algD}
The scope of the present subsection is to study the category of $\de$-algebras. Understandably, the term `differential algebra' would be quite a misnomer, hence our choice to refer to such objects as \emph{derivative algebras}. More explicitly:
\begin{definition}[$\de$-algebras]\label{de_alg_def}
	Let $(\clK,\otimes,\de)$ be a differential 2-rig. The category $\Alg(\de)$ of \emph{derivative algebras} is defined as the category of endofunctor algebras of the derivative $\de$; explicitly, it's the category having
	\begin{itemize}
		\item objects the pairs $(X,\xi)$ where $\xi : \de X \to X$ is a morphism in $\clK$;
		\item morphisms $(X,\xi)\to(Y,\theta)$ the morphisms $f : X \to Y$ in $\clK$ such that $\theta\circ\de f = f\circ\xi$, or in diagrammatic terms
		      \[\vxy{
				      \de X \ar[r]^{\de f}\ar[d]_\xi & \de Y \ar[d]^\theta \\
				      X \ar[r]_f & Y
			      }\]
	\end{itemize}
	where composition and identities are defined as in $\clK$.
\end{definition}
General facts about categories of endofunctor algebras entail that
\begin{remark}
	The category $\Alg(\de)$ admits all limits that $\clK$ admits, and such limits are created by the obvious forgetful functor $U : \Alg(\de)\to\clK : (X,\xi)\mapsto X,f\mapsto f$.

	Colimits in $\Alg(\de)$ are also created by $U$, as long as they are preserved by $\de$; so, if $\de$ preserves coproducts, $U$ creates coproducts, and if $\de$ preserves all colimits, $U$ creates all colimits.
\end{remark}
\begin{corollary}\label{itsa_2_rig}
	If $\clK$ is a $\bfD$-2-rig, the category $\Alg(\de)$ of derivative algebras admits $\bfD$-colimits.
\end{corollary}
\begin{proof}
	In the terminology of \autoref{doctrines_for_all} a $\bfD$-2-rig is nothing but a $T_\bfD M$-pseudoalgebra, from the characterization of derivative algebras as the inserter of $\de$ and $\id[\clK]$ one swiftly deduces that the forgetful functor from $\bfD$-2-rigs to $\Cat$	creates inserters (as well as all other 2-limits), thus giving a slick proof of \autoref{itsa_2_rig}.
\end{proof}

In particular, the main result of this section (\autoref{anchelui_diffrig} below) shows that, remarkably, the category $\Alg(\de)$ is a differential 2-rig thanks to the fact that $\de$ satisfies the Leibniz property. Note that in full generality a sufficient (but, as this example shows, not  necessary) condition ensuring that $\Alg(F)$ is monoidal, and the forgetul functor $U : \Alg(F)\to\clK$ is strict monoidal, is that $F$ is oplax monoidal, and that $\de$ is instead very far from being oplax monoidal.

\begin{theorem}\label{anchelui_diffrig}
	Let $(\clR,\otimes,\de)$ be a differential 2-rig such that $\de I\cong\varnothing$. Define a structure on $\Alg(\de)$ as follows:
	\begin{itemize}
		\item a tensor bifunctor
		      \[\vxy{\firstblank\boxtimes\firstblank : \Alg(\de)\times \Alg(\de)\ar[r] & \Alg(\de)}\]
		      sending two $\de$-algebras $(A,\alpha),(B,\beta)$ to the object $(A,\alpha)\boxtimes(B,\beta):=(A\otimes B,\alpha\boxtimes\beta)$ where
		      \[\label{formica}\vxy{
			      \alpha\boxtimes\beta : \de(A\otimes B) \ar[r]^{\ell^{-1}} & \de  A \otimes B + A \otimes \de B \ar[r]^-{\cop[s]{\alpha\otimes B}{A\otimes\beta}}& A\otimes B
			      }\]
		      is the candidate $\de$-algebra map (and $\ell$ the Leibniz isomorphism of \autoref{def_gen_diff2rig}.\ref{d_2}).
		\item an algebra structure on $I$ given by the initial map $\bang:\de I \cong \varnothing \to I$;
		\item associators and unitors deduced from those in $\clR$.
	\end{itemize}
	Then, $(\Alg(\de),\boxtimes, (I,\bang))$ is a monoidal category. Furthermore, if $(\clR,\otimes,\de)$ is a differential 2-rig for the doctrine of colimits $\bfD$, then $\Alg(\de)$ is a differential 2-rig as well (for the same doctrine of colimits) in such a way that the forgetful functor $U : \Alg(\de) \to \clR$ is a differential 2-rig morphism.
\end{theorem}
\begin{proof}
	Clearly, the condition $\de I\cong\varnothing$ is needed to equip the monoidal unit with a (unique) algebra structure $\de I \cong \varnothing \to I$ and make $(I,!)$ the monoidal unit. The associator and unitor diagrams proving that $\firstblank\boxtimes\firstblank$ is a tensor functor on $\Alg(\de)$ all boil down to the associator and unitor diagrams of the monoidal structure on $\clR$.
\end{proof}
\begin{remark}
	The above construction relies on the identity 2-cell $\de\de\To\de\de$ as a fairly trivial choice of distributive law of $\de$ over itself; there can be other choices: in the category of species, only one is nontrivial: from \autoref{flesh_deriv} $\de = \homDay[\Uno][\firstblank]$, so that with a Yoneda argument one gets that
	\[[\Spc,\Spc](\de\de,\de\de) = [\Spc,\Spc](\homDay[\Due],\homDay[\Due])\cong\sfP(\Due,\Due)\cong S_2\]
\end{remark}
\begin{remark}[The Taylor expansion of an object]\label{the_taylor}
	Note that this procedure for obtaining a differential 2-rig $\Alg(\de)$ of derivative algebras can be iterated, constructing a tower
	\[\label{die_torre}\vxy{
		\Alg(\de'')\ar[d] \ar@{}[r]|-\dots & \Alg(\de'')\ar[d]\\
		\Alg(\de') \ar[d]\ar[r]_-{\de''}& \Alg(\de')\ar[d]\\
		\Alg(\de)\ar[d] \ar[r]_-{\de'}& \Alg(\de)\ar[d]\\
		\clR \ar[r]_-\de & \clR
		}\]
	where each endofunctor action is essentially the one of $\de$; more precisely, $\de'$ acts on a $\de$-algebra sending $(X,\xi : \de X\to X)$ to $(\de X,\de\xi)$; and $\de''$ acts on a $\de'$-algebra $((X,\xi),\varsigma : \de'(X,\xi)\to (X,\xi))$ sending it to $(\de\de X,\de\varsigma)$, where $\varsigma$ is a homomorphism of $\de$-algebras, such that
	\[\vxy{
			\de\de X \ar[d]_{\de \xi}\ar[r]^{\de\varsigma}& \de X \ar[d]^\xi\\
			\de X \ar[r]_\varsigma & X
		}\]
	determining essentially by definition a `second-degree expansion' $X\leftarrow \de X \leftarrow \de\de X$; inductively, then, one determines for every object $X\in\clR$ of a differential 2-rig $(\clR,\de)$ as in \autoref{anchelui_diffrig} a chain
	\[\vxy{
			X &\ar[l] \de X &\ar[l] \de\de X &\ar[l] \cdots &\ar[l] \de^{(n)} X &\ar[l] \cdots
		}\]
	called the \emph{Taylor chain} of $X$. More precisely, one can build the following object inductively:
	\begin{itemize}
		\item $\clR^{(0)} := \clR$ and $\clR^{(n+1)} := \Ins(\de^{(n)}, \clR^{(n)})$ (\ie, the inserter realizing the category of endofunctor algebras for the functor $\de^{(n)}$);
		\item $\de^{(1)} := \de$ and $\de^{(n+1)} := \clR^{(n+1)}\to \clR^{(n+1)}$ defined lifting $\de^{(n)}$.
	\end{itemize}
	The categories $\clR^{(n)}$ and the functors $U^{(n)} : \clR^{(n+1)} \to \clR^{(n)}$ arrange then as the tower in diagram \Cref{die_torre}, and one can then consider the limit of such a chain,
	\[\label{furetto}\Jet[\clR,\de]:=\lim\left(\clR \xot{U} \clR^{(1)} \xot{U^{(1)}} \clR^{(2)} \xot{U^{(2)}}\cdots \right).\]
	Such a category is called the \emph{jet category} of $\clR$, and a typical object in $\Jet[\clR,\de]$ consists of a countable sequence
	\[\vec X = \left( X
		, \Big(X; \prevar[\xi]{\de X}X\Big)
		, \Big((X;\xi);\prevar[\xi']{\de'(X;\xi)}{(X;\xi)}\Big)
		, \dots
		\right)\] the $n^\text{th}$ element of which equips the $(n-1)^\text{th}$ with an algebra structure for $\de^{(n)}$. Morphisms are determined similarly.

	In analogy with differential geometry, where the $k$-jet of a real valued function $f : \bbR\to \bbR$ is defined as
	\[(J^k_{x_0}f)(z)
		=\sum_{i=0}^k \frac{f^{(i)}(x_0)}{i!}z^i
		=f(x_0)+f'(x_0)z+\cdots+\frac{f^{(k)}(x_0)}{k!}z^k.\]
	we define the \emph{$k$-jet} $J^k(\vec X)$ of an object $\vec X \in \Jet[\clR,\de]$ as the image of $\vec X$ under the functor $J^k$ obtained from the limit projections $\pi_k : \Jet[\clR,\de] \to \clR^{(k)}$ as
	\[\vxy{J^k:=\langle\pi_0,\dots, \pi_k\rangle: \Jet[\clR,\de] \ar[r] & \prod_{i=0}^k} \clR^{(i)}\]
\end{remark}
The following result was suggested by C. Chavanat in conversation:
\begin{theorem}
	The category $\Jet[\clR,\de]$ is a 2-rig, and it inherits a differential structure from its universal property, defined as `taking the tail of the sequence $\vec X$':
	\[\label{civetta}\de_\infty\big(X\leftarrow \de X \leftarrow \de\de X\leftarrow\cdots\big) := \big(\de X \leftarrow \de\de X\leftarrow\cdots \big)\]
\end{theorem}
\begin{proof}
	The limit in \Cref{furetto} acquires a natural structure of 2-rig (given the (2-)monadicity that defines a notion of 2-rig, limits are created from $\cate{MCat}$, and subsequently from $\Cat$); we just have to show that $\de_\infty$ as defined in \Cref{civetta} is linear and Leibniz.
	\begin{itemize}
		\item each lifting $\de^{(n)}$ preserves coproducts, that are all computed as in the base $\clR^{(0)}=\clR$. As a consequence, the coproduct of $\vec X=(X,(X;\xi),\dots)$ and $\vec Y=(Y,(Y;\theta),\dots)$ defined as above is
		      \[\left( X+Y
			      , \Big(X+Y; \prevar{\de (X+Y)}{X+Y}\Big)
			      , \Big((X;\xi)+(Y,\theta);\prevar{\de'(X+Y;\#)}{(X+Y;\#)}\Big)
			      , \dots
			      \right).\]
		\item Similarly, the tensor product $\vec X\otimes  \vec Y$ of $\vec X$ and $\vec Y$ in $\Jet[\clR,\de]$ is
		      \[\left( X\otimes Y
			      , (X\otimes Y;\xi\boxtimes\theta)
			      , \Big((X;\xi)\boxtimes(Y,\theta);\prevar{\de'(X\otimes Y;\#)}{(X\otimes Y;\#)}\Big)
			      , \dots
			      \right).\]
		      The functor $\de_\infty$ of \Cref{civetta} acts on this object as follows:
		      \begin{align*}
			      \de_\infty(\vec X\otimes \vec Y) & = \big(\de(X\otimes Y),
			      (\de(X\otimes Y);\de(\xi\boxtimes\theta)),\dots
			      \big)                                                                                               \\
			                                       & \cong \left( \de X\otimes Y
			      , (\de X\otimes Y;\de\xi\boxtimes\theta)
			      , \dots
			      \right) + \left( X\otimes \de Y
			      , (X\otimes \de Y;\xi\boxtimes\de\theta)
			      , \dots
			      \right)                                                                                             \\
			                                       & \cong \de\vec X\otimes \vec Y + \vec X\otimes \de\vec Y.\qedhere
		      \end{align*}
	\end{itemize}
\end{proof}
\subsubsection{Lifting to \textEM algebras}\label{lifta_alle_monadi} Obviously, there exists a (tautological) lifting of $\de$ to a category of Eilenberg-Moore algebras, as the endofunctor $\de$ admits a distributive law with the monads $\de L$ and $R\de$: we expand on this idea in the present section.

To start, recall what it means to lift an endofunctor to an \textEM category of a monad $T$ (dual conditions ensure the lifting to the co\textEM category of a comonad): given $F : \clK\to\clK$ an endofunctor, and $T$ a monad on $\clK$, the following conditions are equivalent:
\begin{enumtag}{la}
	\item there exists a lifting $\hat F : \EM(T)\to\EM(T)$ of $F$ to the \textEM category of $T$;
	\item there exists an endofunctor-to-monad distributive law $\lambda : TF\To FT$, \ie a natural transformation $\lambda$ suitably compatible with the multiplication and unit of $T$.
\end{enumtag}
Now, let $\clR$ be a scopic 2-rig, equipped with a triple of adjoints $L\dashv \de\dashv R$; then the following fact is a general statement about adjoint pairs, applied to $L\dashv \de, \de\dashv R$.
\begin{lemma}\label{biscotto}
	The composite map
	\[\label{papera_1}\vxy{
			\de * (\eta^l \circ\epsilon^l) : \de L \de \ar[r] & \de \ar[r] & \de \de L,
		}\]
	where $L\adjunct{\epsilon^l}{\eta^l} \de$, is an intertwiner of the monad $\de L$ onto itself; similarly, the composite map
	\[\label{papera_2}\vxy{
			(\eta^r \circ \epsilon^r) * \de : \de R \de \ar[r] & \de \ar[r] & R\de\de,
		}\]
	where $\de\adjunct{\epsilon^r}{\eta^r} R$, is an intertwiner of the monad $R\de$ onto itself.
\end{lemma}
\begin{lemma}\label{lupocane}
	The monad $\de L$ on the category of species is commutative.
\end{lemma}
\begin{proof}
	Let's start equipping $T=\de L$ with a tensorial strength of type
	\[\vxy{
			TA\Day B \ar[r] & T(A\Day B)
		}\]
	which considering the isomorphism $T = \id[] + L\de$ can be obtained from the left strength of $L\de$
	\[\label{stre_1}TA\Day B \cong A\Day B + (L\de A)\Day B \xto{1+{\ell'}} A\Day B + L\de(A\Day B)\]
	Similarly, the right strength of $L\de$ gives
	\[\label{stre_2}A\Day TB \cong A\Day B + A\Day (L\de B) \xto{1+{\ell''}} A\Day B + L\de(A\Day B)\]
	One routinely checks that \Cref{stre_1} and \Cref{stre_2} are compatible with the unit and multiplication of $T$, as determined in \autoref{lumaca}.\ref{fw_3}, and then commutativity of $\de L$ amounts to the commutativity of
	\[\vxy{
		& TX\otimes TY \ar@{->}[ld]_\sim \ar@{->}[rd]^\sim &  \\
		X\otimes TY+L\de X\otimes TY \ar@{->}[d]_{1+\ell'} &  & TX\otimes Y + TX\otimes L\de Y \ar@{->}[d]^{1+\ell''} \\
		X\otimes TY+L\de (X\otimes TY) \ar@{->}[d]_{\ell''+L\de\ell''} &  & TX\otimes Y + L\de (TX\otimes Y) \ar@{->}[d]^{\ell'+L\de \ell'} \\
		T(X\otimes Y) + L\de(T(X\otimes Y)) \ar@{->}[rd] \ar@/_/@{->}[rdd]_{\cop[s]{1_{T(X\otimes Y)}}{\epsilon_{T(X\otimes Y)}}} &  & T(X\otimes Y) + L\de(T(X\otimes Y)) \ar@{->}[ld] \ar@/^/@{->}[ldd]^{\cop[s]{1_{T(X\otimes Y)}}{\epsilon_{T(X\otimes Y)}}} \\
		& TT(X\otimes Y) \ar@{->}[d] &  \\
		& T(X\otimes Y) &
		}\label{passero}\]
	which is a compatibility between the two strengths decomposing the leibnizator, induced by the fact that the two identifications $X\otimes TY+L\de X\otimes TY\cong TX\otimes TY\cong TX\otimes Y + TX\otimes L\de Y$ are obtained from the distributivity isomorphisms and the symmetry of $\Day$, compatible with $\ell',\ell''$ in the sense that
	\[\vxy[@C=2cm]{
		{X\otimes Y+X\otimes L\de Y} \ar[r]^-{sw_{XY}+sw_{X,L\de Y}}\ar[d]_{1+\ell''} & {Y\otimes X + L\de Y\otimes X} \ar[d]^{1+\ell'}\\
		{X\otimes Y+ L\de (X\otimes Y)} \ar[r]_-{sw_{XY} + L\de (sw)} & {Y\otimes X + L\de (Y\otimes X)}
		}
	\]
	which boils down to the equations
	\[
		L\de(sw_{XY})\circ \ell'' = \ell'\circ sw_{X,L\de Y}\qquad
		L\de(sw_{XY})\circ \ell'  = \ell''\circ sw_{X,L\de Y}
	\]
	The composite of the two isomorphisms at the top of \Cref{passero}, using again the identification $T\cong \id[]+L\de$, is the morphism
	\[\notag\small\vxy[@R=4mm]{
			X\otimes TY + L\de X\otimes Y \ar[d]^\wr& TX\otimes Y + TX\otimes L\de Y\ar[d]^\wr\\
			X\otimes (Y+L\de Y) + L\de X \otimes (Y+L\de Y) \ar[d]^\wr& (X+L\de X)\otimes Y + (X+L\de X)\otimes L\de Y\ar[d]^\wr\\
			(X\otimes Y + X\otimes L\de Y) + (L\de X\otimes Y + L\de X\otimes L\de Y) \ar[r]_-\sim & X\otimes Y + L\de X\otimes Y + X\otimes L\de Y + L\de X\otimes L\de Y
		}\]
	obtained swapping and rebracketing the two central terms.
\end{proof}
\begin{remark}
	For a left scopic differential 2-rig equipped with a tensor-hom derivative, as in \autoref{tensor_hom_der}, the following generalization of \autoref{lupocane} above holds: if $LI$ is such that the monoid $[LI,LI]$ is commutative, then the monad $\de L$ is commutative.\footnote{This observations is due to Todd Trimble, who provided a graciously elegant proof using Kelly graphs.}
\end{remark}
\begin{remark}
	\label{non_lo_era}
	This seems to be as far as one can get: with \autoref{biscotto}, one lifts $\de$ to a endofunctor $\hat \de$ of $\EM(\de L)$, a category that however is not monoidal (a monoidal structure would come from an \emph{oplax} monoidal structure on $\de L$); on the other hand, $\Kl(\de L)$ is (symmetric) monoidal thanks to the lax monoidal structure of \Cref{stre_1}, \Cref{stre_2} and \autoref{lupocane}, but a lift of $\de$ to $\Kl(\de L)$ would come from a distributive law $\de\de L \To \de L\de$, in the opposite direction of \Cref{papera_1}.
\end{remark}
\subsection{Differential equations and automatic differential equations}\label{sec_diffeq}
Leveraging on the de\-fi\-ni\-tion of Arbogast algebra given in \ref{arbogast_algebra}, one can devise a notion of \emph{differential equation} in a differential 2-rig.

Let's fix a 2-rig $(\clR,\de)$ for the doctrine of coproducts (generalizing to another doctrine is straightforward). Then, a generic element $D$ of $\Arb[\clR,\de]$ is a finite sum $D=\sum_{i\in I} A_i\otimes \de^{n_i}$ that can be considered as an endofunctor of $\clR$, taking $X$ to $DX=\sum_{i\in I} A_i\otimes (\de^{n_i}X)$.

The definition of differential equation we give is motivated by the fact that, in absence of additive inverses, a reasonable way to attach an `equation to solve' to an endofunctor $F$ of a category $\clC$ is to look for its fixpoints, i.e. for objects $X\in\clC$ equipped with an isomorphism $FX\cong X$. As for the (seemingly arbitrary) choice to consider only \emph{maximal} fixpoints of $F$, \ie \emph{terminal $F$-coalgebras}, the following remark (the proof of which is immediate) shows that initial algebras for elements $D\in\Arb[\clR,\de]$ tend to be trivial.
\begin{definition}[Differential equations in a 2-rig]
	Let $(\clR,\de)$ be a differential 2-rig; if the elements of $\Arb[\clR,\de]$ are regarded as differential operators, a \emph{solution} for a differential equation prescribed by $D\in\Arb[\clR,\de]$ is a terminal $D$-coalgebra.
\end{definition}
\begin{remark}
	If $(\clR,\de)$ is a closed, right-scopic 2-rig, then the initial algebra of an element $D\in\Arb[\clR,\de]$ is the initial object of $\clR$. Indeed, the functor $D = \sum_{i\in I} A_i\otimes\de^{n_i}$ is cocontinuous, having the functor $\prod_{i\in I}R^{n_i}[A_i,\firstblank]$ as right adjoint.
\end{remark}
By contrast, the terminal $D$-coalgebra arises as the limit of the op-chain
\[\vxy{
		1 &\ar[l] D1 &\ar[l] DD1 &\ar[l] \cdots
	}\]
prescribed by Ad\'amek theorem \cite{Adamek1974}, for which even for an `affine' differential operator of the form $A\otimes\de(\firstblank)+B$ the computation yields an object that depends on the iterated derivatives of the terminal object.

In the particular case of species, \autoref{term_in_spc} yields that the terminal object is a fixpoint for $\de$, whence the sequence above reduces to
\[\vxy{
		\fkE &\ar[l] A\otimes \fkE +B &\ar[l] A\otimes \de(A\otimes \fkE +B)+B &\ar[l] \cdots
	}\]
With some patience and an inductive argument, each step of the limit can be reduced to an expression in the iterated derivatives of $A,B$.
\subsection{Differential and co/monadic dynamics: $\Mly(\de,B)$}\label{diff_comon_dyn}
Besides monoidal automata in the category $\Mly[\Spc](L,B) = \Mly[\Spc]^\otimes(y\Uno,B)$, 
one can exploit the other adjunction $\de\dashv R$ in which $\de$ sits, and this leads naturally to the study of categories $\Mly[\Spc](\de,B)$ of \emph{differential automata}, where dynamics are induced by the subsequent derivatives of a state object $E,\de E,\dots,\de^n E = E^{(n)},\dots$

Then, from the triple of adjoints $L\dashv\de\dashv R$, a `monad-comonad' and `comonad-monad' adjunction $L\de\dashv R\de$ and $\de L\dashv \de R$ arises.

One can then put the categories $\Mly[\Spc](L\de,B)$ and $\Mly[\Spc](\de L,B)$ under the spotlight using the language of \Cref{sec_abs_auto_in_spc}.	This is what we do in \Cref{LD_automata} below after we address the problem in more generality.

We want to study categories $\Mly(T,B)$ of $(T\dashv S)$-automata where $T$ is a left adjoint monad, and dually, categories $\Mly(Q,B)$ of $(Q\dashv R)$-automata where $Q$ is a left adjoint comonad.

In the case of a left adjoint monad, several technical results can be used to make the description of the categories $\Mly(T,B)$ easier:
\begin{itemize}
	\item \cite[4.3.2]{Bor2} if $T$ is a left adjoint monad, with $S$ as right adjoint comonad, its \textEM category $\clK^T$ is cocomplete, with colimits preserved by the forgetful functor; in fact more is true:
	\item \cite[4.4.6]{Bor2} if $T$ is a left adjoint monad, with $S$ as right adjoint comonad, colimits in $\clK^T$ are \emph{created} by $U$, which in fact is comonadic and $\clK^T$ identifies with the co\textEM category of $S$.
\end{itemize}

The first general observation is completely elementary but already useful: considering that co/monads admit co/unit natural transformations to/from the identity functor, and given the functoriality of $\Mly(-,B)$, we get canonical choices of functors
\[\vxy{
		\Mly(\id,B) \ar[r] & \Mly(Q,B) & \Mly(T,B) \ar[r] & \Mly(\id,B)
	}\]
One can immediately prove from the description of $\Mly(\id,B)$ as a pullback in \Cref{eq_mly_n_mre} that
\begin{remark}
	The category $\Mly(\id,B)$ is the category of coalgebras for the functor $\firstblank\times B$.
\end{remark}
\begin{definition}[Bar and cobar Mealy complexes]\label{bar_n_cobar_cplx}
	Arguing again by (contravariant) functoriality, the monad structure on the functor $T$ specifying the dynamics yields an augmented cosimplicial object
	\[\label{bar}
		\small\vxy{
			\Mly(\bsT,B)_\bullet=\Bigg( \Mly(\id,B) &
			\Mly(T,B) \ar[l]_-{\eta^*}\ar[r]|-{\mu^*}&
			\Mly(T^2,B)\pair{(T\eta)^*}{(\eta T)^*}[l] \pair{}{} & \dots \Bigg)
			\ar[l]\ar@<1em>[l]\ar@<-1em>[l]
		}
	\]
	obtained feeding the bar resolution of $T$ to the functor $\Mly(-,B)$, \cite{goerss-jardine}, \cite[8.6]{Weibel1994}.

	Dually, the cobar resolution of a left adjoint comonad $Q$ yields an augmented simplicial object
	\[\label{cobar}\small\vxy{
			\Mly(\bsQ,B)_\bullet=\Bigg(\dots \pair{}{}&
			\Mly(Q^2,B) \ar[r]|-{\sigma^*} \ar[l]\ar@<1em>[l]\ar@<-1em>[l] &
			\Mly(Q,B) \pair{(\epsilon Q)^*}{(Q\epsilon)^*}[l] &
			\Mly(\id,B) \ar[l]^{\epsilon^*}\Bigg)
		}\]
	%
	We refer to these as the \emph{bar complex} of $T$-automata and the \emph{cobar complex} of $Q$-automata.
\end{definition}
\begin{remark}
	Some intuition on \autoref{bar_n_cobar_cplx} is due. Let $k$ be a field and $A$ be a $k$-algebra, or more generally let $A$ an internal monoid in a monoidal category $(\clK,\otimes,I)$; the bar resolution of the monad $T_A=A\otimes_k\firstblank$ is then useful to compute the Hochschild cohomology of $A$, as the bar complex of $T_A$ is its free resolution as $A$-$A$-bimodule. One generalizes this to an arbitrary monad in the fashion of \cite{Duskin1975,BarrBeck1969} and gets the bar resolution as a `thickening' of $T$ into a simplicial object.
%
\end{remark}
\color{black}
\begin{remark}
	Let $\clK$ be locally presentable. Given that $\mu^* : \Mly(T,B)\to\Mly(T^2,B)$ acts by precomposition with $\mu$, sending $\mlyob Eds$ to $\mlyob E{d\mu_E}{s\mu_E}$ a swift application of the adjoint functor theorem yields a right adjoint $\mu_*$ to $\mu^*$.
\end{remark}
\begin{remark}[On monadic automata]\label{EM_macchine}
	It is reasonable to describe \emph{\textEM}\footnote{The `Moore$_1$' of `Moore automaton' and the Moore$_2$ of `Eilenberg--Moore' are two different people; the notion of `Eilenberg--Moore$_2$ Moore$_1$ automaton' makes perfect sense as a category $\mu\Mre(T,B) := \clK^T\times_\clK\clK/B$. However, we leave Eilenberg--Moore$_2$ Moore$_1$ automata out of this note.} Mealy automata, refining the pullbacks in \autoref{mly_n_mre} by using the forgetful from $\clK^T$ (the Eilenberg--Moore category of $T$) instead of $\Alg(T)$, and obtaining categories $\mu\Mly(T,B)$ and $\mu\Mre(T,B)$ of monadic Mealy and monadic Moore automata; in this case, some of the observations listed here carry over:
	\begin{itemize}
		\item $\mu\Mly(\id,B)$ is just the slice $\clK/B$, so the free-forgetful adjunction $F^T:\clK\rightleftarrows\clK^T : U^T$ induces a `pulled-back' adjunction $\mu\Mly(T,B)\rightleftarrows T/B$.
		\item Let $S,T$ be monads on $\clK$. Whenever a morphism of monads $\lambda : T\To S$ in the sense of \cite[§6.1]{barr+wells1985} is given, the induced (colimit-preserving) functor $\clK^S\to\clK^T$ (cf.~[\emph{ibi}, Thm. 6.3]) induces in turn a (colimit-preserving) functor $\mu\Mly(S,B)\to\mu\Mly(T,B)$.
	\end{itemize}
\end{remark}
\begin{remark}\label{unrewarding}
	Working in the more restrictive case of \textEM automata is, however, rather unrewarding for a variety of reasons: first of all, there is the trivial remark that as soon as a carrier $E$ has a structure $a : TE\to E$ of $T$-algebra, its `dynamics' is pretty trivial, as $a$ must be a split epi with a privileged right inverse $\eta_E$; thus, the composition $s\circ \eta_E$ `knows everything' about the evolution of $\mlyob Eds$. Second, the conditions for a natural transformation to induce functors between \textEM categories are fairly more imposing, and third, the morphisms inducing an analogue of \Cref{bar},\Cref{cobar} are simply not available.

	Something can be said, however, if we work `interfiber' using \autoref{glob_Mly_cat}. A monad morphism in the sense of \cite{Street1972} induces a monad $\hat S$ on $\clK^T$ so that the forgetful $U^T : \clK^T\to\clK$ is an intertwiner, hence leveraging on \autoref{glob_Mly_cat} we can induce a functor
	\[\vxy{\Mly[\clK^T](\hat S,(B,b)) \ar[r] & \Mly(S,B).}\]
\end{remark}
Dually, one can try to render the free functor $F_T : \clK \to \clK_T$ into the Kleisli category of $T$ strong monoidal for a monoidal structure on $\clK_T$; this will yield functors $\Mly(S,B) \to \Mly[\clK_T](\check S,F_TB)$. The matter is investigated in the second part of \cite{Guitart1980} when $F=A\otimes\firstblank$. For example, consider $\clK$ monoidal and with countable sums preserved by the tensor; then, every oplax monoidal monad $T : \clK\to\clK$ lifts a monoidal structure on $\clK_T$ and one can then consider $\clK_T$-valued $F$-machines, cf.~\cite[Prop. 30]{Guitart1980}.
\begin{remark}[On the proper choice of output objects]
	The construction of \autoref{mly_n_mre} depends not only on $F$, but also on an output object $B$, usually thought as a `space of responses' the machine $\mlyob Eds$ can give as output. The choice of what $B$ best models a given problem has to be made each time according to the nature of the problem itself. However, one is almost always led to consider choices of $B$ that are `spaces of truth values', like a Heyting or Boole algebra, or spaces of probabilities, like the closed unit interval $[0,1]$. The co/completeness of $\Mly(F,B)$ and $\Mre(F,B)$ established in \autoref{remark_on_compl_1} entails that all algebraic structures (=all essentially algebraic theories) can be interpreted in such categories, and the nature of $\Spc$ as a presheaf topos entails that the construction of an object of internal real numbers is more or less straightforward. In particular,
	\begin{itemize}
		\item Hadamard Heyting/Boole algebra objects are just species $B : \sfP \to\Set$ which factor through the subcategory $\mathsf{Heyt}$ or $\mathsf{Bool}$, the simplest case being the constant species $\fkB$ at the booleans $\bfB=\{0<1\}$, with trivial action of each $S_n$ ($\fkB$ is the subobject classifier of $\Spc$; another example of a Boolean algebra object in $\Spc$ is the species $\wp$ of subsets of \autoref{exam_species}.\ref{es_1});
		\item regarding $\Spc$ as a presheaf topos, it is easy to determine that the NNO, the object of integers, and of rationals, and of internal Dedekind reals \cite[§VI.1]{mac1992sheaves} can be constructed as constant functors $c_\bbN,c_\bbZ,c_\bbQ,c_\bbR$ at natural, integers, rationals and reals in $\Set$.
	\end{itemize}
\end{remark}

\section{The fourfold way}\label{LD_automata}
The `fourfold way' is the study of the categories $\Mly(L\de,B)$ (in relation with the right adjoint $R\de$ of the dynamics) and $\Mly(\de L,B)$ (in relation with the right adjoint $\de R$). The four functors $L\de,R\de,\de L,\de R$ relate to each other and admit explicit descriptions giving rise to a rich theory of $\de L$- and $L\de$-algebras serving to study the indexed categories $\Mly[\Spc](L\de),\Mly[\Spc](\de L)$.
\begin{remark}[On the structure of $L\de$ and $\de L$]\label{on_DL_and_LD}
	Rajan \cite{rajan1993derivatives} provides explicit formulas for the monads and comonads associated to $L\dashv\de\dashv R$. Let $\fkF  : \sfP \to \Set$ be a species. Then,
	\begin{itemize}
		\item $L\de\fkF $ acts as $y\Uno\Day \de \fkF $; a structure of type $L\de \fkF $ on a finite set $A$ chooses a point of $A$, and an $\fkF $-structure on the complement of that point.
		\item $R\de\fkF $ acts as $A\mapsto \prod_{a\in A}\fkF [(A\smallsetminus\{a\})\sqcup\{\bullet\}]$, i.e. as $A\mapsto (\fkF A)^A$; a structure of type $R\de \fkF $ on a finite set $A$ chooses an $\fkF $-structure on $A$ for every $a\in A$. With a similar reasoning,
		\item $\de L\fkF  = \de(y\Uno\Day \fkF )$ is the functor $\fkF + L\de \fkF$.\footnote{This gives rise to the evocative formula: $[\de,L]=\de L-L\de =1$, i.e. to the canonical commutation relation between position and momentum (up to a sign); in the language of \emph{virtual species} \cite{joyal1985calcul,joyal1985regle,yeh1985combinatorial,412e8ae396fc12c5284630b51a24e521afbc5fc7} and \cite[§2.5]{bergeron1998combinatorial} such an equation can be made completely formal. As for its meaning, \emph{hanc marginis exiguitas non caperet}, but see \autoref{prob_1} below.}
		\item $\de R\fkF$ acts as $A\mapsto \fkF[A]^A \times \fkF[A]=R\de \fkF[A]\times \fkF[A]$.
	\end{itemize}
\end{remark}
\begin{remark}\label{lumaca}
	The structures of each of these four functors intervene in defining the monad structures on $\de L$ and $R\de$ and the comonad structures on $L\de$ and $\de R$:
	\begin{enumtag}{fw}
		\item\label{fw_1} The comultiplication of the comonad $L\de$ has a particularly simple form, being obtained from the coproduct injection
		\begin{align}
			L\de\fkF & \to L\de \fkF + LL\de\de\fkF \notag \\
			         & \cong L(\de\fkF+L\de\de\fkF)\notag  \\
			         & \cong L\de L\de\fkF.
		\end{align}
		\item\label{fw_2} the unit of the monad $R\de$ is a natural transformation with components $\fkF A\to (\fkF A)^A$; this can be taken to be the constant map, \ie the mate of the first projection $\fkF A\times A \to \fkF A$; the multiplication is instead obtained from the (mate of the) counit $\epsilon = \pi_2 : R\de\fkF A \times \fkF A \to \fkF A$ of $\de\dashv R$.
		\item\label{fw_3} The unit of the monad $\de L$ is the first coproduct injection, and the multiplication is induced as
		\[\vxy{
			\de L(\fkF + L\de \fkF)\ar[r]^-\cong & \de L\fkF + \de LL\de \fkF \ar[r]^-{\cop[s] 1{\de L \epsilon}} & \de L\fkF
			}\]
		if $\epsilon$ is the counit of $L\dashv \de$. We observed in \autoref{lupocane} that this monad is commutative.
		\item\label{fw_4} To conlude, the comultiplication of the comonad $\de R$ is obtained from the unit of $\de \dashv R$ as the map with components
		\[\vxy[@R=0mm]{
			\de R \fkF A \ar[r] & \de((R\fkF)^{\id[]})A \\
			R\fkF A \ar[r]_-{\langle\eta_{A+1},1\rangle}& R\fkF[A+1]^A\times R\fkF[A+1],
			}\]
		denoting $R\fkF^{\id[]}$ the functor $A\mapsto (R\fkF A)^A$.
	\end{enumtag}
\end{remark}
The discussion in \autoref{mangusta} yields restrictive assumptions on when a differential 2-rig $(\clR,\de)$, such that $\de$ is a right adjoint with left adjoint $L$, gives rise to a derivation $L\de$.

Recall that the differential operator $\Upsilon = \sum_{i=1}^n x_i\frac{\partial}{\partial x_i}$ in $\bbR^n$ is called `Euler homogeneity operator', cf.~\cite[p. 296]{gelfand_shilov_1968}; another name for the same operation, `numbering derivation', comes from Physics where if $X^n$ represents something like a state of $n$ bosons, like photons in a laser, then the differential operator $X \cdot D$ takes $X^n$ to $n X^n$, where the coefficient `counts' or `numbers' the of bosons.

This leads to the following definition:
\begin{remark}[The Euler derivation on $\Spc$]\label{Euler_der}
	The functor $L\de = y\Uno\Day \de : \Spc\to\Spc$ of \autoref{lumaca}.\ref{fw_1} is a derivation, and furthermore a left adjoint (with right adjoint $R\de$), hence $(\Spc,\Day,L\de)$ is a differential 2-rig for the doctrine of all colimits.
\end{remark}
Armed with the explicit descriptions in \autoref{lumaca}, we can attempt to unveil the structure of the categories $\Alg(L\de)$, $\Alg(\de L)$, as building blocks for the category $\Mly[\Spc]({L\de},B)$, $\Mly[\Spc]({\de L},B)$.
A thorough analysis of co/algebra structures for such interesting endofunctors of $\Spc$ seems to be missing from the existing literature. Rajan \cite{rajan1993derivatives} goes as close as determining in painstaking detail the monad and comonad structures on $\de L,\de R,L\de,R\de$, but doesn't seem to provide a characterization for their endofunctor or \textEM algebras, or even for the (much easier, and somewhat more inspiring) bare endofunctor algebras.
As one would expect from the adjunction relations $L\de\dashv R\de$ and $\de L\dashv \de R$ the structures of $L\de$-algebras (=$R\de$-coalgebras) and $\de L$-algebras (=$\de R$-coalgebras) are tightly related.
The following computations all follow a general argument, given \autoref{on_DL_and_LD} a $\de L$-algebra structure on a species $F$ consists of a pair $\cop[s] uv : F + L\de F \to F$ of maps $u : F \to F$ and $v : \de F \to \de F$ of endomorphisms, one for $F$ and one for $\de F$.
\begin{example}
	A $\de L$-algebra structure on the exponential species $E$ reduces to a pair $u:E\to E$ and $v:LE\to E$, which in turn reduces to another endomap of $E$, given how $E$ is a Napier object. Then, $\de L$-algebra structures on $E$ are representations of the free monoid $\bbN\langle d,c\rangle$ (cf.~\cite{guitart2014I,guitart2014II,guitart2017III}) on 2 generators $d,c$ over the set $E\Uno$ (because endomaps of $E$ are in bijection with elements of $E\Uno$, by Yoneda). For set species, this must be trivial, for linear species this amounts to a `character' for the monoid representation $\bbN\langle d,c\rangle$.
\end{example}
\begin{example}
	For the species $\Lin$ of linear orders, a $\de L$-algebra map is a map $L\otimes L \to L$, since
	\[\Lin + L\de (\Lin) = \Lin + y\Uno \Day \Lin \Day \Lin\]
	but then $\Lin+y\Uno \Day \Lin \Day \Lin=\Lin \Day(1+y\Uno\Day\Lin)$, and the fact that $1+y\Uno\Day\Lin\cong\Lin$ is exactly the universal property satisfied by $\Lin$ as initial algebra of $1+y\Uno\Day\firstblank$.
\end{example}
\begin{example}
	A similar line of reasoning leads to the characterization of $\de L$-algebra structures on the species of cycles, \autoref{exam_species}.\ref{es_4}: since $\de \Cyc\cong\Lin$, structures of $\de L$-algebras are pairs, $\Cyc \to \Cyc$ and $\Lin \to \Lin$ of endomorphisms.
\end{example}
\begin{example}
	For the species $\fkS$ of permutations of \autoref{exam_species}, a $\de L$-algebra structure consists of a pair $\cop[s] uv :\fkS+L\de\fkS\to\fkS$, where $v$ can in turn be simplified into $\fkS\Day(1+y\Uno\Day\Lin)\cong\fkS\Day\Lin$ using \autoref{ex_alg_4}.
\end{example}

\section{Other flavour of species}\label{other_flavs}
The purpose of this section is to schematically extend the results of the paper (with particular attention to \autoref{its_diffe}, \autoref{sec_diffeq}), to categories that arise as generalisations of $\Spc$ or that exhibit similar properties than the ones of $\Spc$.

For example, when the set $S$ in \autoref{spc_n_Vspc} has more than one element, we get \emph{coloured} species \cite[§1.1]{Joyal1986foncteurs}, \cite[2.1.4, (6)]{7a8211434eed2811547338107aa8e1aa26e0ff5f}; when instead of bijective functions of finite sets we take \emph{injective} functions $\Inj$, the presheaf category $[\Inj^\op,\Set]$ is species-like; it's the category of nominal sets of \cite{Pitts2013}, also known as the \emph{Schanuel topos}; when instead of finite sets and bijections we consider finite \emph{ordinals} and monotone bijections we get a very rigid domain category for linearly ordered species \cite{10.1007/BFb0072518}\dots
\subsection{Coloured species}\label{colo_spc}
% ...
The free symmetric monoidal category on a set $S$ (regarded as a discrete category), as defined in \autoref{spc_n_Vspc}, admits an explicit description in the following terms.

Think of $S$ as a set of \emph{colours} (the terminology comes from operad theory, cf.~\cite{Yau2016_jn}); the category $\sfP[S]$ has objects the finite sets $[n] := \{1,\dots,n\}$ equipped with a function $c : [n] \to S$ called a \emph{coloration} or \emph{colouring function}, and morphisms the bijections $\sigma : [n]\to [n]$ `compatible with the coloration', in the sense that they induce an indexed family of bijection $\sigma_s : [n]_s \to [n]_s$ among the fibers $[n]_s = c^{-1}s$. In other words, $\sfP[S]$ is the comma category
\[\vxy{
		\sfP[S] \ar[r]\ar[d]& 1 \ar[d]^S\\
		\sfP \ar[r]& \Set\ultwocell<\omit>{}
	}\]
Presented in this way, $\sfP[S]$ is a \emph{coloured PROP}: colourings can be tensored using the universal property of sums of finite sets:
\[([n],c)\oplus([m],d) = \left([n+m], \cop[s]cd\right)\]
where $\cop[s]cd : i\mapsto [i\le n] \mathbin{\texttt{?}} c(i) \mathbin{\texttt{:}} d(i)$.

As a consequence, the presheaf category of $(S,\Set)$-species (or $S$-coloured species, or $S$-species for short) of \autoref{spc_n_Vspc} acquires a Day convolution monoidal structure.

An explicit description of the (binary, and $n$-fold by induction) convolution of $S$-species $M,N : \sfP[S] \to \Set$ is needed in order to equip $S\emdash\Spc$ with a plethystic substitution structure; this is given in \cite[2.1]{7a8211434eed2811547338107aa8e1aa26e0ff5f}: the coend that defines the convolution splits as the sum
\[\label{chic_fact} M\Day N : ([n],c) \mapsto \sum_{p,q\vdash n} M([p],c)\times N([q],d)\]
where $p,q\vdash n$ denotes the set of decompositions $([p],c),([q],d) \in \sfP[S]$ of $[n]$, \ie the pairs of objects of $\sfP[S]$ such that $([n],c)=\cop[s]pq : [p+q] \to S$.
\begin{remark}\label{decompo}
	Note how this decomposition is possible as a consequence of the fact that $(S,\Set)\emdash\Spc$ splits as a product of groupoids $\sfP[S]\equiv \prod_{s\in S}\sfP$ (in a similar fashion $\sfP$ splits as a product of groups in \Cref{repre_repre}).
\end{remark}

In order to motivate the definition of substitution product, let's review the analogue operation for formal power series. Let $I,J$ be two sets, considered as sets of indeterminates; fix a ring of coefficients $k$, and let $f\in k[J]$, and $\vec g  = \{g_j\in k[I]\mid j\in J\}$ denote, respectively, a polynomial in the indeterminates of the set $J$, and a family of $|J|$ polynomials in the indeterminates $I$.

For all such $f,\vec g$ one defines the \emph{substitution} of $\vec g$ into $f$ as the image of $f$ under the ring homomorphism $k[J] \to k[I]$ uniquely determined by $j\mapsto g_j$: then one has
\[f\lhd \vec g = \sum_{\vec n\in \bbN^J}\vec g^{\,\vec n}\]
where $\vec g^{\,\vec n} := \prod_{j\in J}g_j^{n_j}$ (the notation is motivated by the fact that $c : [n]\to J$ equals $\big[\sum_{j\in J} n_j\big]\to J$). Similarly, substitution of a $J$-family of $S$-species $\vec M = \{M_j\mid j\in J\}$ in a $J$-species $N$ is defined thinking of $\vec M$ as a functor $J \to S\emdash\Spc$ which, by virtue of the universal property of $\sfP[J]$ extends to a unique strong monoidal functor $\vec M^\star : (\sfP[J],\oplus)\to (S\emdash\Spc,\Day)$. This functor sends an object $\vec w = j_1\dots j_n$ to the convolution of the tuple $M_{j_1},\dots, M_{j_n}$; denote $\vec M^*_{\vec w}$ the action of this extension.

Now, the substitution of a $J$-family of $S$-species $\vec M$ into an $S$-species $N$ is defined as
\[\label{color_sub}N\lhd \vec M := \vec u \mapsto \int^{\vec w\in \sfP[J]}N[\vec w]\times \vec M^*_{\vec w}[\vec u]\]
Given the structure decomposition of $\sfP[J]$, one sees that the coend in question splits into the sum of coends
\[\sum_{r\ge 0}\int^{(j_1,\dots,j_r)\in\sfP[J]} N[(j_1,\dots, j_r)]\times \vec M^*_{(j_1,\dots,j_r)}[\vec u]\]
where $\vec w = (j_1,\dots, j_r)$ is a generic object of $\sfP[J]$.

Once the notation that is required for the proof is settled, the following result boils down to a painstakingly long coend computation:
\begin{proposition}
	The substitution product in \Cref{color_sub} defines a bifunctor
	\[\vxy{J\emdash\Spc\times (S\emdash\Spc)^J \ar[r] & S\emdash\Spc}\]
	which is `associative and unital' in the sense that the (relatively obvious) diagrams must commute. Unwinding the conditions prescribed by these diagrams yields that:
	\begin{itemize}
		\item given a single $K$-species $F$, a $K$-family of $J$ species $\vec G = \{G_k\mid k\in K\}$, and a $J$-family of $S$-species $\vec H = \{H_j\mid j\in J\}$, then
		      \[F\lhd (\vec G\lhd \vec H)\cong (F\lhd\vec G)\lhd \vec H\]
		      where on the LHS we define $\vec G\lhd \vec H$ as the $J$-family of species $\{G_k\lhd \vec H\mid k\in K\}$;
		\item there exist a $J$-family of $J$-species $\vec y_r$ acting as right unit for $\lhd$, \ie such that $N\lhd \vec y\cong N$, naturally in the $J$-species $N$, and for every $j\in J$ there exists a $J$-species $y_j$ such that $y_j\lhd\vec M\cong M_j\in \sfP[S]$ for every $J$-family of $S$-species.
	\end{itemize}
\end{proposition}
\begin{theorem}
	The category of $S$-species admits an Hadamard product (given by the Cartesian structure on $\Cat(\sfP[S],\Set)$) and a Cauchy product (Day convolution) given by \Cref{chic_fact} above. Moreover, it admits \emph{partial derivative} functors $\frac\de{\de s}$ that `derive along a colour' $s\in S$, and satisfy the commutation rule $\frac \de{\de s}\frac \de{\de t}\cong \frac \de{\de t}\frac \de{\de s}$ for each $s,t\in S$.
\end{theorem}
\begin{proof}
	To fix ideas, let $S=\{s,t\}$ have just two elements; extending the following argument to an arbitrary set $S$ is obvious. Define
	\[\textstyle\frac \de{\de s} : S\emdash\Spc\overset{\Cref{decompo}}\cong \Cat(\sfP[\{s\}]\times\sfP[\{t\}],\Set)\cong\Spc^{\sfP[\{t\}]} \xto{\de_*}\Spc^{\sfP[\{t\}]}\cong S\emdash\Spc,\]
	and $\frac \de{\de t}$ similarly. With this definition, the square
	\[\label{part_deriv}\vxy{
		S\emdash\Spc \ar[d]_{\frac\de{\de s}}\ar[r]^{\frac\de{\de t}}& S\emdash\Spc\ar[d]^{\frac\de{\de s}}\\
		S\emdash\Spc \ar[r]_{\frac\de{\de t}} & S\emdash\Spc
		}\]
	can be decomposed into commutative sub-squares in a straightforward way.
\end{proof}
From \Cref{part_deriv} it is evident how each $\frac\de{\de s}$ admits a left and a right adjoint, thus giving to $S\emdash\Spc$ the structure of a scopic differential 2-rig $(S\emdash\Spc,\frac \de{\de s})$ for each choice of colour $s\in S$. As a consequence, all theorems that apply to a scopic differential 2-rig apply to $S\emdash\Spc$.
\subsection{$k$-vector species}\label{sec_linspc}
\begin{definition}[Vector species]\label{def_lin_specs}
	Let $k$ be a field and $S$ a set; the category of $k$-vector $S$-species is the category of $(S,k\emdash\Vect)$-species in the sense of \autoref{spc_n_Vspc}.
\end{definition}
The category of $(1,k\emdash\Vect)$-species is simply called the category of ($k$-)vector species. Similar definitions hold more generally for $R$-modules, but vector species are a more interesting subject for enumerative combinatorics due to a highly nontrivial theory of Hopf monoids under Day convolution, cf.~\cite{aguiar2010monoidal}, that has (evidently) relations to the linear representations of the symmetric groups. For the purposes of the present work, moving to vector species maintains all the core results, and enriches some. In particular, vector species carry the same monoidal structures of \autoref{various_mono} (a fact that is the basic building block for the interest of combinatorialists and algebraists/geometers in vector species, cf.~\cite{aguiar2010monoidal,loday2012algebraic}), they form a scopic differential 2-rig (with a similar argument of \Cref{futtanota}). As a consequence, all theorems that apply to a scopic differential 2-rig apply to $(S,k\emdash\Vect)\emdash\Spc$.
\subsection{Linearly ordered species}\label{sec_ordspc}
The category $\cLin$ is defined in \cite{10.1007/BFb0072518} as the category of totally ordered finite sets $\langle n\rangle := \{1<\dots<n\}$ and \emph{order-preserving} bijections $\sigma : \langle n\rangle \to \langle n\rangle$. Let's give a more intrinsic presentation for it.
\begin{definition}
	Let $S_n$ be the symmetric group of an $n$-set $[n]$. Let $r : BS_n \to\Set$ be the (functor associated to the) left regular representation of $S_n$, \ie the action $S_n\to S_n$ given by left multiplication; denote $[S_n\slice S_n]$ the associated action groupoid \cite{Higgins1971}. \ie the strict pullback
	\[\vxy{
			[S_n\slice S_n] \ar[r]\ar[d]\ar@{}[dr]|(.375){\dopb} & \Set_* \ar[d]\\
			BS_n \ar[r]_r & \Set.
		}\]
\end{definition}
\begin{remark}\label{its_trivial}
	Notice that since the action is strictly transitive, $[S_n\slice S_n]$ consists of the maximally connected groupoid on the underlying set of $S_n$. As such, the unique functor $[S_n\slice S_n]\to 1$ is an equivalence of categories.
\end{remark}
\begin{definition}[The $\bbL$ category, and $\bbL$-species]
	We define the category $\bbL$ as the coproduct (in the category of groupoids) $\sum_{n\ge 0} [S_n\slice S_n]$; if $\clV$ is a Bénabou cosmos, the category of $\clV$-valued $\bbL$-species is the category of functors $\bbL \to \clV$.
\end{definition}
In the following we use the shorthand of denoting the category of $\Set$-valued $\bbL$-species simply as $\bbL\Spc$. As a consequence of the equivalence established in \autoref{its_trivial} above, an $\bbL$-species is \emph{essentially} a symmetric sequence:
\[\textstyle\bbL\Spc = \prod_{n\ge 0}\Cat([S_n\slice S_n],\Set)\cong \prod_{n\ge 0}\Cat(1,\Set)\cong\Set^\bbN.\]

However, the interest in $\bbL$-species arises as (contrary to what happens for $(S,\Set)\emdash\Spc$, cf.~\cite[§2.5]{bergeron1998combinatorial}) differential equations in $\bbL\Spc$ have unique solutions [\emph{ibi}, §5.0], following more closely the properties of formal power series.

Usually one compensates for the extreme rigidity of the domain category of an $\bbL$-species fixing a commutative ring $A$ and `enriching' the codomain of species in the category of \emph{$A$-weighted sets} (cf.~\cite[§2.3]{bergeron1998combinatorial}).

Although $\bbL\Spc$ is not a category of the form $(S,\clV)\emdash\Spc$ in the sense of \autoref{spc_n_Vspc}, it is a `species-like' category, in the sense that it retains similar properties of the ones enjoyed by $\Spc$:
\begin{itemize}
	\item Similarly to \autoref{some_ups_P}.\ref{ups_2}, $\bbL$ is the skeleton of the category $\cLin$ of finite totally ordered sets, and order-preserving bijections (here, relabeling functions can exist between sets whose elements have different `names');
	\item limits, colimits
	\item $\cLin$ carries a monoidal structure given by \emph{ordinal sum}, cf.~\cite[§5.1]{bergeron1998combinatorial}, defined as $\langle n\rangle\oplus\langle m\rangle := \{1<\dots<n<1'<\dots<m\}$; as a consequence, $\bbL\Spc$ has a Day convolution $\Day^\bbL$ monoidal structure, and a plethystic substitution operation, similarly to \autoref{various_mono}.\ref{ms_3}.
	\item $\bbL\Spc$ is a differential 2-rig under the derivative functor $\de \fkF\langle n\rangle := \fkF\langle 1\oplus n\rangle$ (the new element $1$ is adjoined to $\langle n\rangle$ as a \emph{bottom} element).
	\item $\bbL\Spc$ is equipped with structures that are not present, or behave worse, in $\Spc$: for example, it carries a \emph{Heaviside} product,\footnote{Called `convolution product' in \cite[§5.1]{bergeron1998combinatorial}, a terminology that might clash with `Day convolution'.} defined as
	      \[F\oast G := F\Day^\bbL y\Uno\Day^\bbL G,\]
	      and an \emph{antiderivative} operation $\int F$, defined (on objects of $\cLin$) as
	      \[\textstyle (\int F)\varnothing = \varnothing \qquad\qquad (\int F) U = F(U\setminus\{\min U\})\]
	      where $\min U$ is the bottom element of $U$. Observe that $\de(\int F)=F$.
\end{itemize}
\begin{remark}\label{LSPC_scopic}
	Note that $\de$ defined above admits a left and a right adjoint, with a similar argument as the one given in \Cref{futtanota}; this makes $\bbL\Spc$ a scopic differential 2-rig.
\end{remark}
As a consequence, all theorems that apply to a scopic differential 2-rig apply to $\bbL\Spc$.
\subsection{M\"obius species}\label{sec_mobspc}
\begin{definition}\label{fPosTopbotIso}
	Let $\Set$ be the category of sets, equipped with the tautological functor $J:\Set\to\Cat$ regarding each set as a discrete category; let $\Pos^{\top\bot}$ be the category of posets with top and bottom, where morphisms are top- and bottom-preserving monotone maps; consider the comma category 
	\[\label{category_int}\vxy{
		(J/\Pos^{\top\bot})\ar[r]\ar[d] & {*}\ar[d]^{\Pos^{\top\bot}} \\
		\Set \ar[r]_J & \Cat. \ultwocell<\omit>{}
		}\]
\end{definition}
Unwinding the definition, $(J/\Pos^{\top\bot})$ is the category having
\begin{itemize}
	\item objects the pairs $(X,P : X \to \Pos^{\top\bot})$ where $X$ is a set, and $P$ is a functor: note that this means $P = \{P_x\mid x\in X\}$ is a $X$-parametric family of posets with top and bottom;
	\item morphisms $(X,P)\to(Y, Q)$ are the functions $h : X\to Y$ such that $Q\circ h = P$. Each such $h$ splits into a family of monotone maps $P_x \to Q_{hx}$;
\end{itemize}
\begin{remark}
	The category $(J/\Pos^{\top\bot})$ is a complete and cocomplete (in fact, locally presentable), monoidal closed category.
\end{remark}
\begin{proof}
	Colimits are computed in $(J/\Pos^{\top\bot})$ as in $\Set$ (created by the vertical left functor in \Cref{category_int}); the category is accessible, as it arises as a limit in accessible categories and accessible functors; thus it is locally presentable, hence also complete (limits are, however, not straightforward to describe --even characterizing a terminal object is a bit convoluted).

	As for its monoidal closed structure, call a map $P\times Q\to R$ in $\Pos^{\top\bot}$ \emph{balanced} if all $f(p,\firstblank) : Q\to R$ and $f(\firstblank,q) : P\to R$ preserve top and bottom elements, and call $\cate{BPos}(P\times Q,R)$ the set of all such balanced maps.
	
	Then, the existence of a symmetric tensor product $\hat\otimes$ such that
	\[\cate{BPos}(P\times Q,R)\cong \Pos^{\top\bot}(P\mathbin{\hat\otimes} Q,R)\]
	`representing balanced maps' follows from a standard argument on lifting monoidal structures to categories of algebras ($\Pos^{\top\bot}$ is the category of algebras for the simultaneous completion under initial and terminal object, regarding $\Pos\subset\Cat$).

	This gives a monoidal (closed) structure to $(J/\Pos^{\top\bot})$ where tensor and exponentials are defined as
	\begin{gather}
		X\times Y \xto{P\times Q} \Pos^{\top\bot}\times \Pos^{\top\bot} \xto{\hat\otimes} \Pos^{\top\bot},\notag\\
		X\times Y \xto{P\times Q} (\Pos^{\top\bot})^\op\times \Pos^{\top\bot} \xto{[\firstblank,\firstblank]} \Pos^{\top\bot}.\qedhere
	\end{gather}
\end{proof}
Thus, $((J/\Pos^{\top\bot}), \hat\otimes)$ works as Bénabou cosmos, and we can define
\begin{definition}[M\"obius species]\label{def_mobius_specs}
	The category $\sfM\Spc$ of \emph{M\"obius species} is the category of functors $\sfP \to (J/\Pos^{\top\bot})$, \ie the category $(1,(J/\Pos^{\top\bot}))$-species in the notation of \autoref{spc_n_Vspc} and \autoref{fPosTopbotIso}.
\end{definition}
Since $\sfP$ is a groupoid, each functor $\sfP\to(J/\Pos^{\top\bot})$ must factor through the core of $(J/\Pos^{\top\bot})$; calling `$\Int$' such core we obtain \cite[Definition 2.1]{mendez1991mobius} where $h$ is assumed to be a bijection (and the indexing sets are finite, hence $h : [n] \to [n]$ is just a permutation), inducing order-isomorphisms $P_i \cong Q_{\sigma i}$ for each $i = 1,\dots,n$.

Now, the definition of Day convolution, plethystic substitution, and derivative are as in \autoref{sec_species}, just changing base of enrichment for $(J/\Pos^{\top\bot})$; the derivative endofunctor $\de : \sfM\Spc \to \sfM\Spc$ has a left and a right adjoint, thus making $(\sfM\Spc,\Day,\de)$ into a scopic differential 2-rig.

As a consequence, all theorems that apply to a scopic differential 2-rig apply to $\sfM\Spc$.
\subsection{Nominal sets}
\begin{definition}
	Consider the chain of inclusions
	\[S_1 \subset S_2\subset \dots \subset S_n \subset \dots \]
	each identifying a group $S_n$ as the subgroup of $S_{n+1}$ spanned by the elements fixing $n+1$; the colimit $S_\infty$ of this chain in the category of groups is called the \emph{infinite symmetric group} and consists of all bijections of $\bbN = \{0,1,2,\dots\}$ that fix all but finitely many elements (call these \emph{finitely supported permutations}).
\end{definition}
The group-theoretic properties of $S_\infty$ are the subject of intense study in connection with representation theory \cite{Dixon1996,vershik2012nonfree}, the theory of Von Neumann algebras \cite{thoma1964unzerlegbaren,vershik2002characters,okounkov1996young}, ergodic theory, \cite{olshanski1991unitary,glasner2003ergodic} and descriptive set theory \cite{Kechris1995} (due to the nature of Polish group of $S_\infty$). For us, the connection with computer science \cite{Pitts2013,Petrisan2010}, set theory \cite{Wraith1978,Felgner1971,Blass1992} and topos theory \cite{fiore1999abstract,FioreMenni2004} are an additional source of intuition: we define
\begin{definition}[Nominal species]\label{def_nominal_specs}
	The category $\Nom$ of \emph{nominal sets} is the category of (set-theoretic) left actions of $S_\infty$, or in other words the category of functors $F : S_\infty\to\Set$.
\end{definition}
So, the category of nominal sets is the topos of $S_\infty$-sets. There are equivalent descriptions for $\Nom$: among them, what we refer to as the \emph{Schanuel model} for nominal sets, we characterize
\begin{itemize}
	\item $\Nom$ as the category of sheaves for the atomic topology on $\Inj^\op$ of finite sets and \emph{injections};
	\item $\Nom$ as the category of pullback preserving functors $\Inj\to\Set$.
\end{itemize}
The category of nominal sets is especially important in light of its relation to the category of species: Fiore and Menni observed that the obvious inclusion $i:\sfP\hookrightarrow \Inj$ induces a left adjoint monad $T_i=i^* \circ \Lan_i : \Spc \to \Spc$, and $\Nom$ identifies to $\Kl(T_i)$.
\begin{proposition}
	The category $\Nom$ admits a Day convolution monoidal structure $\Day$ (regarding it as presheaves over $S_\infty$ its existence follows from general facts about convolution of $G$-sets; in the Schanuel model, one mimics the definition of \autoref{some_ups_P}.\ref{ups_4} for injective functions -given injections $[n]\to [n'],[m]\to [m']$ there is an injection $[n+n']\to [m+m']$- and gets the same expression of \Cref{day_def}, just sheafified).
	As a consequence, the category $(\Nom,\Day)$ is a 2-rig for the doctrine of all colimits.

	$\Nom$ also admits a plethystic substitution operation induced by $\Day$ and defined similar to \autoref{various_mono}.\ref{ms_3}. Day convolution is a closed monoidal structure, and it can be shown (cf.~\cite{0b3f85c56e6c7e0598d42984000f394b12f187fd}) that $\homDay[y\Uno][F]$ is a derivative with a left adjoint $y\Uno\Day\firstblank$ ($y\Uno$ is the \emph{sheaf} associated to the representable on $\Uno$, which is not already a sheaf, cf.~\cite[p. 1]{Barr1980}).

	As a consequence, all theorems that apply to a \emph{left} scopic differential 2-rig apply to $\Nom$; incidentally, note that $\Nom$ is an example of a left scopic differential 2-rig which is not scopic, as $\de = \homDay[y\Uno]$ can't have a right adjoint (it doesn't preserve all colimits).
\end{proposition}
\section{Conclusions and future work}
Here we sketch directions of investigation for the future.
\begin{problem}\label{prob_1}
Let $\bfK$ be a strict 2-category with all finite weighted limits. Consider objects $X,B\in\bfK$ in a diagram of the following form:
\[\label{vaca}\vxy{
		K \ar[r]_-\id & K & \ar[l]^-f K \ar[r]_-f & K & B \ar[l]^-b
	}\]
The \emph{Vaucanson limit} \cite{Heudin2008-ne}
\footnote{Jacques de Vaucanson ($\ast$1709--$\dag$1782) was, besides the inventor of the modern lathe and of automatic loom, the creator of sophisticated and almost lifelike mechanical toys such as the `\emph{flûteur automate}' and the `\emph{canard défécateur}'. The mechanical duck appeared to have the ability to eat kernels of grain, and to metabolize and defecate them.}
obtained from \Cref{vaca} consists of the limit obtained taking (cf.~\cite{Fiore2006,2catlimits})
\begin{itemize}
	\item the \emph{inserter} $K \xot u \clI(f,\id[K])\xto u K$ of the left cospan;
	\item the \emph{comma object} $K \xot v f/b \xto q B$ of the right cospan;
	\item the strict pullback $\clI(f,\id[K])\times_K (f/b)$ of $u,v$.
\end{itemize}
If $\bfK$ is the 2-category $\mathbf{Cat}$ of categories, functors, and natural transformations, Vaucanson limits recover the categories $\Mly(A,B)$ when $B=1$ is the terminal category and $b$ is an object therein.

Formal theory of Mealy automata is then the study of Vaucanson objects in $\bfK$. One can define analogues for $\Mly(A,B)$, $\Mre(A,B)$ enriched over a generic monoidal base $\clW$ in the sense of \cite[Ch. 6]{Bor2}, \cite{kelly}, for example a quantale \cite{rosenthal1,Eklund2018} like $[0,\infty]^\op$, so that there is a \emph{metric space} $\Mly[(X,d)](f,b)$ \cite{Lawvere1973,Clementino2003,hofmann2014monoidal} associated to every nonexpansive map $f : X\to X$ and point $b\in X$. This begs various questions: what is this theory (intended in the technical sense: \Cref{vaca} can be pruned to become the simple diagram
\[\vxy{ \ar@(ul,dl)_{f} K & B\ar[l]^-b }\]
and interpreted as a certain $\Cat$-enriched limit sketch of which categories $\Mly(f,b)$ are models: this is suggestive, in light of \cite{lmcs_6213}), and how can it profit from being studied via discrete dynamical methods? Can it be related with fixpoint theory as classically intended in \cite{Granas2003-yj}?
\end{problem}
\begin{problem}\label{prob_2}
The canonical commutation $[\de,L]=\de L-L\de =1$ valid in Joyal's virtual species suggests how $L$ acts as a `conjugate operator' to $\de$. Compare this with the analogue relation $[x\cdot\firstblank,\frac d{dx}]=1$ valid in the ring $C^\omega(\bbR)$ of analytic functions on, say, the real line \cite[Ch. 5]{Eyges1980ez}, \cite{folland2009fourier}. Is it the case that there is a still undiscovered `categorified Greenfunctionology' introducing a `Heaviside distribution' $\Theta$ with the property that the colimit of $F$ weighted by $\Theta$ is a solution of the differential equation $\de G = F$ on species, i.e. $\de\left(\int^X\Theta(X,\firstblank)\times F[X]\right) \cong F$?
Compare this with the well-known integral equation $\frac d{dx}\left(\int \Theta(x-t)f(t)dt\right)=f(x)$ for the Heaviside function, and cf.~\cite{day2011} where Day sketched a categorified theory of Fourier transforms (upper and lower transforms, Parseval relations, etc.) for categories enriched over a $*$-autonomous base $\clV$ \cite{Barr1979}, generalizing Joyal's categories of analytic functors.
We intend to pursue the matter, captivated by its compelling aesthetic beauty.
\end{problem}

\putbib{bib/allofthem}{amsalpha}
\end{document}